\documentclass[9pt,shortpaper,twoside,web]{ieeecolor}
\usepackage{generic}
\usepackage{cite}
\usepackage{amsmath,amssymb,amsfonts}
\usepackage{algorithmic}
\usepackage{graphicx}
\DeclareMathAlphabet{\mathpzc}{OT1}{pzc}{m}{it} 
\usepackage{braket}

\usepackage{amsthm}
\usepackage{thmtools}
\usepackage{mathtools}
\usepackage{bigints} 	
\usepackage{steinmetz}  
\usepackage{graphics}  
\usepackage{graphicx}
\let\labelindent\relax
\usepackage{enumitem}
\usepackage{booktabs} 
\usepackage{setspace}
\usepackage{upgreek}
\usepackage{fancyhdr}
\usepackage{float}
\usepackage{tcolorbox}
\usepackage{longtable}
\usepackage{comment}
\usepackage{empheq, blkarray}
\usepackage{subfig}
\usepackage{subfloat}
\usepackage{array}
\newcolumntype{L}{>{\centering\arraybackslash}m{5cm}}
\usepackage{bbding} 
\usepackage{pifont} 
\usepackage{centernot} 

\usepackage{todonotes}

\usepackage{ulem}
\usepackage{nicefrac}

\usepackage{letltxmacro}
\LetLtxMacro{\OldSqrt}{\sqrt}
\newcommand{\ClosedSqrt}[1][\hphantom{3}]{\def\DHLindex{#1}\mathpalette\DHLhksqrt}
\makeatletter
\newcommand*\bold@name{bold}
\def\DHLhksqrt#1#2{%
	\setbox0=\hbox{$#1\OldSqrt{#2\,}$}\dimen0=\ht0\relax%
	\advance\dimen0-0.2\ht0\relax
	\setbox2=\hbox{\vrule height\ht0 depth -\dimen0}%
	{\hbox{$#1\expandafter\OldSqrt\expandafter[\DHLindex]{#2\,}$}
		\lower\ifx\math@version\bold@name0.6pt\else0.4pt\fi\box2}
}
\renewcommand*{\sqrt}[2][\ ]{\ClosedSqrt[\leftroot{-2}\uproot{1}#1]{#2}\kern0.1em} 
\makeatother

\makeatletter
\newcommand{\removelatexerror}{\let\@latex@error\@gobble}
\makeatother
\usepackage{verbatim}
\tcbuselibrary{breakable}
\newcommand\numberthis{\addtocounter{equation}{1}\tag{\theequation}}

\theoremstyle{definition}
\newtheorem{definition}{Definition}
\newtheorem{assumption}{Assumption}
\theoremstyle{plain}
\declaretheorem[name=Theorem]{thm}
\declaretheorem[name=Lemma]{lem}
\declaretheorem[name=Proposition]{prop}
\declaretheorem[name=Corollary]{cor}

\theoremstyle{remark}
\newtheorem{remark}{Remark}
\newtheorem{example}{Example}
\usepackage{framed} 

\DeclareMathOperator{\dom}{dom}

\DeclareMathOperator{\diag}{diag}
\DeclareMathOperator{\Jacob}{\mathbf{J}}

\DeclareMathOperator{\interior}{int}
\DeclareMathOperator{\rinterior}{rint}

\DeclareMathOperator{\closure}{cl}

\DeclareMathOperator{\range}{ran}

\DeclareMathOperator{\E}{\mathbb{E}}
\DeclareMathOperator{\J}{\mathbf{J}}
\DeclareMathOperator{\gJ}{\boldsymbol{\mathfrak{J}}}

\DeclareMathOperator{\I}{\mathpzc{I}}

\DeclareMathOperator{\arcsinh}{arcsinh}

\usepackage{caption}
\renewcommand\thetable{\arabic{table}}
\usepackage[colorlinks=true,
linkcolor=black,
urlcolor=black,
citecolor=black]{hyperref}

\begin{document}
\title{Second-Order Mirror Descent: Convergence in Games Beyond Averaging and Discounting}
\author{Bolin Gao and Lacra Pavel%
	\thanks{This work was supported by a grant from NSERC and Huawei Technologies Canada.}
	\thanks{B. Gao and L. Pavel are with the Department of Electrical and Computer Engineering, University of Toronto, Toronto, ON, M5S 3G4, Canada. Emails:
	{\tt\small bolin.gao@mail.utoronto.ca, pavel@control.utoronto.ca}}%
}

\maketitle

\IEEEpeerreviewmaketitle
\begin{abstract} In this paper, we propose a second-order extension of the continuous-time game-theoretic mirror descent (MD) dynamics, referred to as MD2, which provably converges to mere (but not necessarily strict) variationally stable states (VSS) without using common auxiliary techniques such as time-averaging or discounting. We show that MD2 enjoys no-regret as well as an exponential rate of convergence towards strong VSS upon a slight modification. MD2 can also be used to derive many novel continuous-time primal-space dynamics. We then use stochastic approximation techniques to provide a convergence guarantee of discrete-time MD2 with noisy observations towards interior mere VSS. Selected simulations are provided to illustrate our results.  
\end{abstract}

\section{Introduction}

A central problem of multi-agent online learning is the design of adaptive policies which a set of subscribing agents (or players) can utilize to arrive at a desired collective outcome. These adaptive policies can typically be viewed as the following iterative process: an \textit{information processing step}, whereby game-relevant information is made available to the player and then collected for processing, followed by a \textit{decision step}, whereby the processed information is converted into the next strategy. A unified mathematical codification of this two-step process that emerged in recent years is referred to as \textit{mirror descent}, first studied by \cite{Nemirovsky83} and subsequently by \cite{Beck03, Doan, Krichene15, Duchi, Yu20, Tsianos, Raginsky12, Nesterov09} in the context of optimization. 

For \textit{games}, mirror descent is often analyzed as a set of ordinary differential equations (ODEs), referred to as the continuous-time mirror descent dynamics (MD), which can be seen as the multi-agent extension of \cite[p. 87]{Nemirovsky83}. MD is also referred to as dual averaging (DA) in continuous games \cite{MZ2019, Staudigl17} and Follow-the-Regularized-Leader (FTRL) or exponential learning in finite games \cite{Flokas, Mertikopoulos16, Mertikopoulos18}. The core appeal of MD lies in its flexibility in adapting to a wide variety of problem settings as well as its rich theoretical properties. In particular, MD is known to converge to the Nash equilibrium (NE) in terms of its generated strategies under a trio of \textit{strict} assumptions:
\begin{itemize}
	\item[(i)]  the game is \textit{strictly monotone}, i.e., the pseudo-gradient of the game is a strictly monotone operator\cite{Facchinei_I}. The class of strictly monotone games captures games that admit strictly concave potential functions, as well as saddle-point problems with strictly convex, strictly concave saddle functions, a special case of diagonally strict concavity of \cite{Rosen65}.
	\item[(ii)]  the NE is \textit{strictly variationally stable} (in the sense of \cite{MZ2019}). The class of games with a strictly variationally stable NE (also known as \textit{strictly variationally stable games} \cite{Zhou_Lossy_18}) contains the class of strictly monotone games, as well as strictly coherent saddle point problems \cite{Optimistic_MD}. 
	\item[(iii)] the NE is \textit{strict} \cite{Harsanyi}, which coincides with a locally strictly variationally stable NE in finite games \cite{MZ2019}. 
\end{itemize} 

Despite these theoretical guarantees, the applicability of MD and its variants remains limited, as many games found in practice do not conveniently exhibit these strict properties. In other words, MD does not converge in many games, most notably in zero-sum (ZS) finite games that feature a unique interior NE \cite{Mertikopoulos18}. \cite{Hofbauer_Stable} showed that no two-player ZS finite game is strictly monotone. Furthermore, strictness is also intimately related to the uniqueness of the game equilibrium. For instance, a strictly monotone game admits a unique NE \cite{Facchinei_I}. In practice, however, many games exhibit a convex set of equilibria as opposed to a unique one. This means strictness imposes serious constraints on the problem parameters or types. Moreover, many dynamics or algorithms that are directly derived from MD (through discretization or otherwise) share similar limitations in that some type of strictness needs to be assumed to ensure convergence, e.g., \cite{MZ2019, Mertikopoulos16, MigotCojocaru2020, Zhou_Counter, Zhou_Lossy_18, Staudigl17, Flokas, Giannou21}. 

To achieve convergence beyond strictness in a continuous-time setup, the existing literature typically relies on two primary approaches: \textit{averaging} and \textit{discounting}. The first method utilizes the time-averaged (\textit{ergodic}) strategies generated by MD as opposed to the actual strategies \cite{Mertikopoulos16}. However, this approach has two drawbacks. Firstly, the players need to combine their time-averaged strategies in order to recover the NE, which is unrealistic in non-cooperative settings. Secondly, time-averaging could fail outside of ZS games \cite{Unstable_Equilibria}. Another method is via a \textit{discounting} procedure \cite{Bo_LP_TAC2020, Bo_Pavel_TechNote_2021, Coucheney}, which is conceptually related to the \textit{weight decay} method of \cite{Krogh92}. Discounting can be seen as a regularization technique that makes use of a strictly convex function to offset poor game properties, such as the lack of strict monotonicity. However, in general, discounting cannot yield exact convergence \cite{Coucheney, Bo_LP_TAC2020, Bo_Pavel_TechNote_2021}. Since \textit{averaging} and \textit{discounting} represent two of the most widely studied method for improving the convergence properties of continuous-time MD, yet at the same time both have been met with challenges, therefore it is natural for us to set our sights on alternative approaches.

\textbf{\textit{Contributions}}:
In this work, in a game-theoretic setup, we propose a second-order variant of the continuous-time MD (the first-order of which was studied in \cite{Raginsky12, Krichene15, Staudigl17, Mertikopoulos16, Mertikopoulos18, Flokas}), which we refer to as MD2. \textit{Second-order} means that the set of ODEs is second-order in \textit{time}, thus MD2 can be seen as a dual-space formulation of the heavy-ball method \cite{Polyak} and is closely related to the class of $n$th order discounted game dynamics of \cite{Bo_LP_TAC2020}, adapted here in a more general game setting. Like MD, we show that MD2 satisfies the basic assumption of \textit{no-regret}. However, unlike MD, MD2 converges without using averaging even when the NE is not strictly variationally stable. Moreover, unlike the discounted MD \cite{Bo_Pavel_TechNote_2021}, which partially overcomes the convergence issue of MD at the cost of inexact convergence, MD2 converges \textit{exactly}. Exact convergence via a primal-space, second-order pseudo-gradient-type dynamics has been achieved in (merely) monotone games \cite{Dian_LP_CDC2020}. We provide a dual-space generalization of this result, beyond the monotone game setting, into the case where the equilibria are \textit{merely variationally stable}. Furthermore, MD2 can exactly recover the unconstrained, primal-space dynamics of \cite{Dian_LP_CDC2020} in the full-information case. Finally, we use our continuous-time convergence results to guide the proof of convergence of a \textit{discrete-time} MD2. In gist, we show that higher-order dynamics can converge towards Nash solutions with more general stability properties. 

\textbf{\textit{Related Works}}: This paper incorporates several ideas from the variational stability, higher-order dynamics, and stochastic (semi-bandit) Nash-seeking literature. The concept of a variationally stable state (VSS) has roots in evolutionary game theory, and it is motivated by and analogous to that of an evolutionarily stable state (ESS) \cite{Smith}. Namely, an ESS is a locally strict VSS for games with one or more populations of agents, also known as population games (similarly, a global ESS or GESS is a globally strict VSS) \cite{Hofbauer_Stable}. Many evolutionary game dynamics such as the replicator and projection dynamics have been shown to converge towards NEs of strictly stable games, which are GESS \cite{Hofbauer_Stable, MigotCojocaru2020}. In contrast to ESS, VSS is defined for the more general class of continuous games \cite{MZ2019}. The literature on strict VSS seeking without structural property such as strict monotonicity is relatively recent \cite{MZ2019, Staudigl17, Zhou_Counter, Zhou_Lossy_18, Flokas}. \cite{Staudigl17} showed that DA converges towards a globally strict VSS. \cite{Zhou_Lossy_18, Zhou_Counter} studied online gradient descent (discrete MD with projection map) where the feedback is received asynchronously between agents.

In contrast to a strict VSS, the class of \textit{mere VSS} is motivated by the more general notion of a neurally stable state (NSS) and global NSS (GNSS) \cite{Smith82}. Comparatively speaking, the set of literature that deals with convergence towards a mere VSS is fewer, especially in the absence of structural assumptions such as the game being merely monotone (also known as stable game \cite{Hofbauer_Stable}). In a population game setup, \cite{Hofbauer_Stable} showed that the best response, integrable excess payoff and impartial pairwise comparison dynamics converge globally to the set of NEs in (merely) stable games, which are GNSS. With a few exceptions such as \cite{Optimistic_MD}, the current set of literature on the (non-ergodic) convergence towards a mere VSS outside of a finite/population game setup typically also assumes monotonicity of the game \cite{Shanbhag12, Tatarenko19, Dian_LP_CDC2020}. In contrast to the above references, in this work we generalize the convergence behavior of MD to non-strict NE (i.e., a mere VSS) without any structural assumption, through the use of higher-order augmentation.

Higher-order dynamics was pioneered by Polyak in \cite{Polyak}, whose algorithm later became widely known as the \textit{heavy-ball method} (HB) \cite{Polyak}. HB was recognized as an analogue of a second-order ODE upon its inception and has since been heavily investigated in the optimization literature \cite{Attouch2020} and arguably forms the backbone of \textit{deep learning}, where non-convexity and local minima dominate \cite{Sutskever}. Inspired by \cite{Polyak}, the authors of \cite{Flam} studied a second-order dynamics for continuous concave games, which can be seen as a second-order extension of the projection dynamics of \cite{Nagurney} and showed that their dynamics converge whenever there exists a Lipschitz and bounded potential function. In a finite game setup, \cite{Laraki13} studied $n$th-order variant of exponential learning \cite{Mertikopoulos16}, whereby the payoff is processed via successive integration, and showed that these higher-order learning schemes can achieve faster rate of elimination of iteratively dominated solutions and convergence towards strict NE. An important connection between higher-order dynamics and the stability of a NE was established by \cite{Shamma_anticipatory}, which showed that the addition of an \textit{anticipatory} process to first-order gradient-play and replicator dynamics can result in local convergence towards an interior (non-strict) NE, despite the dynamics themselves being \textit{uncoupled}. For continuous-time equilibrium seeking via a higher-order augmentation in the dual-space, the closest work related to ours is the second-order \textit{discounted} exponential learning scheme in \cite{Bo_LP_TAC2020}, which overcomes the problem with non-convergence towards interior NE of \cite{Laraki13} at the cost of inexact convergence. It was found empirically in \cite{Bo_LP_TAC2020} that such higher-order augmentation improves the convergence property of its first-order counterpart and converges in non-strictly monotone games where the first-order fails. We improve \cite{Bo_LP_TAC2020} by generalizing these results from finite games to continuous games and go beyond a monotone game setup while simultaneously achieving exact convergence. Exact convergence (also known as \textit{last iterate convergence}) via a primal-space, second-order pseudo-gradient dynamics was achieved in merely monotone games \cite{Dian_LP_CDC2020}, which we also generalize, by stripping away monotonicity assumptions and showing that MD2 encapsulates the dynamics of \cite{Dian_LP_CDC2020} as a special case.  

In the discrete-time, semi-bandit setup, whereby each player receives a (possibly) noise-corrupted version of the pseudo-gradient, the closest work related to ours is \cite{MZ2019}. \cite{MZ2019} studied the first-order MD, known as DA therein, and showed that under imperfect gradient setup, DA converges to a globally variationally stable NE under diminishing step-sizes and its ergodic average converges to the set of NE in 2-player ZS games. Several variations to DA have been studied in the semi-bandit literature, whereby their convergence is towards either strict (strong) VSS or strict NE. For instance, in finite games, \cite{Cohen17} studied a Hedge-variant of exponential weight algorithm (HEDGE) and provided convergence when the NE is strongly VS with respect to the $L^1$ norm. In a similar setup, \cite{Coucheney} studied a discounted variant of HEDGE for potential games. A similar analysis was also performed by \cite{Giannou21} for FTRL in finite games. \cite{Shanbhag13} provided several primal-space algorithms for merely monotone games. \cite{Optimistic_MD} studied the optimistic MD algorithm, which converges in coherent games, i.e., two-player games that possess mere VSS. Compared to optimistic MD, discrete-time MD2 only requires one mirror projection instead of two prox projections. Furthermore, optimistic MD cannot converge in the presence of noise \cite{Optimistic_MD} unless the game is strictly coherent (a stronger assumption), whereas the discrete-time MD2 can converge under noisy conditions.

\textbf{\textit{Organization}}: We provide the background materials in Section II. Section III reviews several basic properties of the first-order MD. We discuss the convergence properties of MD2 in Section IV and provide the rate of convergence as well as  regret minimization guarantee in Section V.  We derive associated primal-space dynamics in Section VI. Section VII discusses MD2 in discrete-time with noisy observations. Simulations are presented in Section VIII. Section IX presents our conclusions. For readability, all proofs are relegated to Section X. A table of main notations used in this paper (\autoref{table:1}) is provided in the Appendix.

\vspace{-0.1cm}
\section{Background}

The following material for convexity and duality are drawn from \cite{Beck17,Wets09}. Those for game theory are drawn from \cite{Basar, Facchinei_I, Ozdaglar}. 

\label{sectionbackground} 

\subsection{Sets and Vectors}
Given a convex set $\mathcal{C} \subseteq \mathbb{R}^n$, the (relative) interior of the set is denoted as ($\rinterior(\mathcal{C})$) $\interior(\mathcal{C})$. 
$\rinterior(\mathcal{C})$ coincides with $\interior(\mathcal{C})$ whenever $\interior(\mathcal{C})$ is non-empty. The closure of a set $\mathcal{C}$ is denoted as $\closure(\mathcal{C})$. We denote the non-negative orthant of $\mathbb{R}^n$ as $\mathbb{R}^n_{\geq 0}$, and the positive orthant as $\mathbb{R}^n_{>0}$. A column vector in $\mathbb{R}^n$ is denoted as $x = (x_1, \ldots, x_n) = \begin{bmatrix} x_1, \ldots, x_n \end{bmatrix}^\top$. $\mathbf{1}$ and $\mathbf{0}$ denote the column vector all ones and all zeros.   $\mathpzc{I}_{n \times n}$ and $\mathpzc{O}_{n \times n}$ denote the $n \times n$ identity and zero matrices. Subscript is omitted when the dimensionality is unambiguous. Suppose $n$ is a natural number, then $[n] \coloneqq \{1, \ldots, n\}$.

\subsection{Convex Functions and Duality}

Let $\mathcal{M} = \mathbb{R}^n$ be endowed with norm $\|\cdot\|$ and dot product $\langle \cdot, \cdot \rangle$. An extended real-valued function is a function $f$ that maps from $\mathcal{M}$ to $[-\infty, \infty]$. The (effective) domain  of $f$ is $\dom(f) = \{x \in \mathcal{M} : f(x) < \infty\}$.
A function $f: \mathcal{M} \to [-\infty, \infty]$ is proper if $f(x) \neq -\infty, \forall x$ and there exists at least one $x \in \mathcal{M}$ such that $f(x)  < \infty$; $f$ is closed if its epigraph is closed. A function $f: \mathcal{M} \to [-\infty, \infty]$ is supercoercive if $\lim_{\|x\| \to \infty} f(x)/\|x\| \to \infty$. The indicator function over $\mathcal{C}$ is denoted by $\delta_{\mathcal{C}}$. The normal cone of $\mathcal{C}$ is defined as, $N_\mathcal{C}(x) =\{v \in \mathbb{R}^n| v^\top(y-x) \leq 0, \forall y \in \mathcal{C}\}$. The tangent cone of a non-empty convex set $\mathcal{C}$ at $x \in \mathcal{C}$ is $T_{\mathcal{C}}(x) = \overline{\bigcup_{\lambda \in \mathbb{R}_{>0}} \lambda (\mathcal{C} - x)}$ and equals $\emptyset$ for all $x \notin \mathcal{C}$. $Id$ denotes the identity function. 

Recall that $\pi_{\mathcal{C}}(x) =  {\text{argmin}}_{y \in \mathcal{C}} \|y - x\|_2^2$ is the Euclidean projection of $x$ onto $\mathcal{C}$, where we refer to $\pi_\mathcal{C}$ as the Euclidean projector. 
Let  $\partial f(x)$ denote a the set of subgradients of $f$ at $x$ and  $\nabla f(x)$ the gradient of $f$ at $x$. For differentiation on the boundary of a closed set $\mathcal{C}$, in lieu of the subgradient, we can also assume $f$ is defined and differentiable on an open set containing $\mathcal{C}$. Given $f$, the function $f^\star \!: \!\mathcal{M}^\star \! \to \! [-\infty, \infty]$ defined by $f^\star(z) \!=\! \sup_{x  \in \mathcal{M}} \! \big[x^\top z \!- \!f(x)\big]$, 
is called the conjugate function of $f$, where  $\mathcal{M}^\star\!$ is the dual space of $\mathcal{M}$, endowed with the dual norm $\| \cdot\|_\star $.   $f^\star$ is closed and convex if $f$ is proper. By the conjugate subgradient theorem \cite{Beck17}, suppose $f: \mathbb{R}^n \to(-\infty, +\infty]$ is proper, closed and convex and $f^\star$ is its Fenchel conjugate, then for any $x, z \in \mathbb{R}^n$, $x^\top z = f(x) + f^\star(z)  \Longleftrightarrow z \in \partial f(x)   \Longleftrightarrow x \in \partial f^\star(z).$ The Bregman divergence of a proper, closed, convex function $f$ is $D_f : \dom(f) \times \dom(\partial f) \to \mathbb{R}, D_f(x,y) =  f(x) -  f(y) -  g^\top(x - y), g \in \partial f(y)$.  Given a vector-valued function $F$, the Jacobian of $F$ is denoted as $\mathbf{J}_F$.

\subsection{N-Player Concave Games}

Let $\mathcal{G} = (\mathcal{N}, \{\Omega^p\}_{p \in \mathcal{N}}, \{\mathcal{U}^p\}_{p\in \mathcal{N}})$ be a game, where  $\mathcal{N} = \{1, \ldots, N\}$ is the set of players, $\Omega^p \subseteq \mathbb{R}^{n_p}$ is the set of player $p$'s strategies (actions). We denote the strategy (action) set of player $p$'s opponents as $\Omega^{-p} \subseteq \prod_{q \in \mathcal{N}, q \neq p} \mathbb{R}^{n_q}$. We denote the set of all the players strategies as $\Omega =   \prod_{p \in \mathcal{N}} \Omega^{p} \subseteq \prod_{p \in \mathcal{N}} \mathbb{R}^{n_p} = \mathbb{R}^{n}, n = \sum_{p \in \mathcal{N}} n_p$. We refer to  $\mathcal{U}^p: \Omega \to \mathbb{R}, x \mapsto \mathcal{U}^p(x)$ as player $p$'s real-valued payoff function, where $x  = (x^p)_{p \in \mathcal{N}} \in \Omega$ is the action profile of all players, and $x^p \in \Omega^p$ is the action of player $p$. We also  denote $x$ as $x  = (x^p;x^{-p})$ where $x^{-p} \in \Omega^{-p}$ is the action profile of all players except $p$. For differentiability purposes, we make the implicit assumption that there exists some open set, on which $\mathcal{U}^p$ is defined and continuously differentiable, such that it contains $\Omega$.

\begin{assumption}\label{assump:blanket} For all $p \in \mathcal{N}$, $\Omega^p$ is a non-empty, closed, convex, subset of $\mathbb{R}^{n_p}$. $\mathcal{U}^p(x^p;x^{-p})$ is (jointly) continuous in $x = (x^p;x^{-p})$. $\mathcal{U}^p(x^p;x^{-p})$ is concave and continuously differentiable ($\mathcal{C}^1$) in $x^p$ for all $x^{-p} \in \Omega^{-p}$. 
\end{assumption} \vspace{-0.23cm}
 Under \autoref{assump:blanket}, we refer to $\mathcal{G}$ as a \textit{(continuous) concave game}. Given $x^{-p} \in \Omega^p$, 
each agent $p \in \mathcal{N}$ aims to find the solution of the following optimization problem,  
\begin{equation}
	\begin{aligned}
		& \underset{x^p}{\text{maximize}}
		& &  \mathcal{U}^p(x^p; x^{-p})
		& \text{subject to}
		& & x^p \in \Omega^p.
	\end{aligned}
\end{equation}
A profile 	${{x}}^\star = ({{x}^p}^\star)_{p \in \mathcal{N}}$  is a Nash equilibrium (NE) if, 
\begin{equation}\label{eqn:Nash_definition}
	\mathcal{U}^p({x^p}^\star; {x^{-p}}^\star) \geq \mathcal{U}^p(x^p; {x^{-p}}^\star), \forall x^p \in \Omega^p, \forall p \in \mathcal{N}.
\end{equation}
At a NE, no player can increase its payoff by unilateral deviation. 

\begin{remark}\label{assump:1} Under Assumption \ref{assump:blanket}, existence of a NE is guaranteed for bounded $\Omega^p$ \cite[Theorem 4.4]{Basar}. 
When $\Omega^p$ is closed but not bounded, the existence of a NE is guaranteed under the additional assumption that  $-\mathcal{U}^p$ is coercive in $x^p$, that is, 
$	\textstyle \lim_{\|x^p\|\to\infty} -\mathcal{U}^p(x^p;x^{-p}) = +\infty,$ 
for all $x^{-p} \in \Omega^{-p}, p \in \mathcal{N}$ \cite[Corollary 4.2]{Basar}. A NE is said to be \textit{strict} when the inequality in \eqref{eqn:Nash_definition} is strict \cite{Harsanyi}. 
\end{remark} 
A useful characterization of a NE of a concave game $\mathcal{G}$ is given in terms of the \textit{pseudo-gradient} \cite{Rosen65}, which is defined  as, \begin{equation} U: \Omega \to \mathbb{R}^n, U(x) = (U^p(x))_{p \in \mathcal{N}}, \end{equation} where 
$U^p(x) =\nabla_{x^p} \mathcal{U}^p(x^p;x^{-p})$ is the \textit{partial-gradient}.
By \cite[Proposition 1.4.2] {Facchinei_I}, $x^\star \in \Omega$ is a NE iff,
\begin{equation}
	(x-x^\star )^\top U(x^\star) \leq 0, \forall x \in \Omega. 
	\label{eqn:nash_equilibrium_VI}
\end{equation}	
Equivalently, $x^\star$ is a solution of the Stampacchia variational inequality 
$\text{VI}(\Omega, -U)$,  \cite{Facchinei_I}. Following \cite{MigotCojocaru2020}, we say that a NE is globally strict if the inequality of \eqref{eqn:nash_equilibrium_VI} is held strictly for all $ x^\star \neq x$.

\subsection{Monotonicity}
A general class of games in which many dynamics are guaranteed to converge is the class of monotone games, also known as stable games \cite{Hofbauer_Stable, Sandholm} or dissipative games \cite{Sorin16}. We contrast some known definitions associated with monotone games in the literature. 
\begin{definition}
	\label{def:monotone_games}
	The game $\mathcal{G}$ is, 
	\begin{itemize}[noitemsep,topsep=0pt]
		\item[(i)] $\eta$-strongly monotone if 
		$		\!(U(x) \!-\! U(x^\prime))\!^\top\!(x\!-\!x^\prime) \!\leq \! -\eta\|x \!-\! x^\prime\|^2_2$, $\! \forall x,x^\prime  \!\in \!\Omega$, for some $\eta \!>\! 0$.
		\item[(ii)] strictly monotone if
		$		(U(x) \!-\! U(x^\prime))^{\!\top}\!(x\!-\!x^\prime) \!<\! 0, \forall x \in \Omega\backslash\{x^\prime\}$, with equality if and only if $x = x^\prime$. 
		\item[(iii)] (merely) monotone if 
		$		(U(x) - U(x^\prime))^\top(x-x^\prime) \leq 0, \forall x, x^\prime \in \Omega. $
		\item[(iv)] $\mu$-weakly monotone if 
		$		\!(U(x) \!-\! U(x^\prime))\!^\top\!(x\!-\!x^\prime) \!\leq \! \mu\|x \!-\! x^\prime\|^2_2$, $\! \forall x,x^\prime  \!\in \!\Omega$, for some $\mu \!>\! 0$.
	\end{itemize} 
\end{definition}
Strongly and strictly monotone games have been extensively investigated in the literature, at least dating from \cite{Rosen65}. Merely monotone games have been studied in \cite{Shanbhag12, Shanbhag13, Tatarenko19, Bo_Pavel_TechNote_2021, Bo_LP_TAC2020, Dian_LP_CDC2020, Hofbauer_Stable}. Weakly monotone games were considered in \cite{Bo_LP_TAC2020, Persis2020, Matt_Pavel}. When $U$ is $\mathcal{C}^1$, there exists a  natural characterization in terms of definiteness of its Jacobian \cite[p. 155, Prop. 2.3.2]{Facchinei_I}. Note that strictly monotone games can have at most one NE, whereas NEs can form a non-singleton convex set in merely monotone games \cite{Facchinei_I}.  

\subsection{Variational Stablility}
Although many classical examples of games satisfy monotonicity properties \cite{Hofbauer_Stable}, these conditions may not hold in more complex scenarios. A recent line of research has started to relax monotonicity notions from a game to that of an equilibrium via the notion of a \textit{variationally stable} (VS) equilibrium or state \cite{MZ2019}.  
\begin{definition} A Nash equilibrium $x^\star \in \Omega$ is, 
	\begin{itemize}[noitemsep,topsep=0pt]
		\item[(i)]  $\eta$-strongly VS if $U(x)^\top(x-x^\star) \leq  -\eta \|x- x^\star\|_2^2, \forall x \in \Omega,$  for some $\eta \!>\! 0$. 
		\item[(ii)] strictly VS if $U(x)^\top(x-x^\star) \leq 0, \forall x \in \Omega$, with equality if and only if $x = x^\star$. 
		\item[(iii)] (merely) VS if $U(x)^\top(x-x^\star) \leq 0, \forall x \in \Omega.$
		\item[(iv)] $\mu$-weakly VS if $U(x)^\top(x-x^\star) \leq  \mu \|x- x^\star\|_2^2, \forall x \in \Omega,$  for some $\mu \!>\! 0$.
	\end{itemize}
	If a condition (i) - (iv)  hold on $\mathcal{D} \subset \Omega$, then the associated definition is said to hold \textit{locally}. Otherwise the definition is said to hold globally. 
	\label{def:VS}
\end{definition}
 
\begin{remark} 
	\label{remark:weak_vss} 
We refer to $x^\star$ as a globally strong/strict/mere/weak \textit{variationally stable state} (VSS) if it satisfies one of the corresponding VS notions in \autoref{def:VS} on all of $\Omega$; $x^\star$ is a locally strong/strict/mere/weak VSS otherwise. For practical reasons, we will informally refer $\mu$-weak VSS with a small $\mu$ as a \textit{nearly} mere VSS, i.e., a weak VSS within a small distance away from becoming a mere VSS.
\end{remark}

Globally mere VSS are the solutions to Minty variational inequality \cite{Facchinei_I}. In a population game context, the set of globally mere VSS is called GNSS and (the unique) globally strict VSS is called GESS \cite{MigotCojocaru2020, Hofbauer_Stable, Sandholm}. Strict VSS was extended to a local/global set-wise definition  in \cite{MZ2019}. \cite{Zhou_Counter} introduced a slight variation of strict VSS called $\lambda$-VS. \cite{Mazumdar} studied a local version of strict VSS under the name \textit{locally asymptotically stable differential NE}, which requires twice continuous-differentiability of the payoff functions.  Strong VSS (\autoref{def:VS}(i)) was also studied in \cite{MZ2019}. \cite{Cohen17} studied a variant of the strong VSS with $L^1$ norm.

It can be seen that any NE of the strongly/strictly/merely/weak monotone game is a strong/strict/mere/weak VSS. In particular, the non-empty set of NE coincides with the set of mere VSS for merely monotone concave games \cite[Theorem 2]{MigotCojocaru2020}. When $\interior(\Omega) \neq \emptyset$ and $x^\star$ is an interior NE, i.e., $U(x^\star) = \mathbf{0}$, the condition for VS recovers the condition associated with monotone at $x^\prime = x^\star$. Hence VS can be seen as a type of point-wise monotonicity and it is known to have a second-order characterization similar to that for monotone games, but specifically at the NE, see \cite{MZ2019}. 

Recall that a NE $x^\star$ is \textit{interior} if it lies in the interior of $\Omega$, that is, $x^\star \in \rinterior(\Omega)$. Throughout this work, we make the following assumption.
\begin{assumption}
	$\mathcal{G}$ admits an interior mere VSS.
\end{assumption}

\subsection{Second-Order Characterization of VSS}
In practice, outside of monotone games, it is often difficult to verify that a NE is a particular type of VSS directly through \autoref{def:VS}. A more common approach to characterize the VSS conditions of a solution is by looking at the \textit{symmetric game Jacobian} whenever the pseudo-gradient $U$ is $\mathcal{C}^1$, which is given by,
\begin{equation} \label{def:game_jacobian} \gJ_U(x) \coloneqq \frac{1}{2}(\Jacob_U(x) + \Jacob^\top_U(x)), x \in \Omega. \end{equation} 
The  matrix $\gJ_U(x) \in \mathbb{R}^{n \times n}$ is symmetric and thus guaranteed to have real eigenvalues, which makes it amenable to analysis. 

We now provide several sufficient conditions for verifying whether a VSS $x^\star$ is strict, mere or weak. This is performed by checking the definiteness of $\gJ_U(x)$ or $\gJ_U(x^\star)$. Our condition for strict VSS is the same as that in \cite{MZ2019}. Before proceeding, we say that the (symmetric) game Jacobian $\gJ_U(x) \in \mathbb{R}^{n \times n}$ \eqref{def:game_jacobian} is negative definite on $T_{\Omega}(x)$ if,
\begin{equation}
	y^\top \gJ_U(x)y  < 0, \forall  x \in \Omega, \forall y \in T_{\Omega}(x), y \neq \mathbf{0},
\end{equation}
and negative semi-definite on $T_{\Omega}(x)$ if the preceding inequality is non-strict. 
We use the shorthand notations $\gJ_U(x) \prec \mathpzc{O}$ and  $\gJ_U(x) \preceq \mathpzc{O}$ respectively. For $\mu > 0$,  $\gJ_U(x) - \mu \mathpzc{I} \prec \mathpzc{O}$ is written as $\gJ_U(x) \preceq \mu \mathpzc{I}$. Similar conventions for when $x = x^\star$.

\begin{prop}
	\label{prop:second_order_char}
	Let $x^\star \in \Omega$ be a NE of $\mathcal{G}$. Suppose $U$ is continuously differentiable, and,
	\begin{enumerate}[align=parleft, labelsep=0.65cm]
		\item[(i)] $\gJ_U(x) \prec \mathpzc{O}$ on $T_{\Omega}(x), \forall x \in \Omega$, then $x^\star$ is globally strictly VS and isolated. 
		\item[(ii)] $\gJ_U(x) \preceq \mathpzc{O}$  on $T_{\Omega}(x), \forall x \in \Omega$, then $x^\star$ is globally merely VS. 
		\item[(iii)] $\gJ_U(x) \preceq \mu \mathpzc{I}$  on $T_{\Omega}(x), \forall x \in \Omega$, then $x^\star$ is globally $\mu$-weakly VS. 
	\end{enumerate}
	Suppose instead, 
	\begin{enumerate}[align=parleft, labelsep=0.65cm]
		\item[(i')]  $\gJ_U(x^\star) \prec \mathpzc{O}$  on $T_{\Omega}(x^\star)$, then $x^\star$ is locally strictly VS and isolated.
		\item[(ii')] $\gJ_U(x^\star) \preceq \mathpzc{O}$  on $T_{\Omega}(x^\star)$, then $x^\star$ is locally merely VS. 
		\item[(iii')] $\gJ_U(x^\star) \preceq \mu \mathpzc{I}$  on $T_{\Omega}(x^\star)$, then $x^\star$ is locally $\mu$-weakly VS. 
	\end{enumerate}
\end{prop}

\begin{remark}
	These conditions can be verified by calculating the maximum eigenvalue of $\gJ_U(x)$ (resp. $ \gJ_U(x^\star)$), where  $\lambda_\text{max}(\gJ_U(x))$ (resp. $\lambda_\text{max}(\gJ_U(x^\star))$) denotes the (real) eigenvalue with the largest magnitude associated with an eigenvector in $T_{\Omega}(x)$ (resp. $T_{\Omega}(x^\star)$). When $\mathbf{J}_U(x)$ (resp. $ \mathbf{J}_U(x^\star)$) is symmetric, then analysis can be performed directly on $\mathbf{J}_U$ without resorting to calculating $\gJ_U$. 
\end{remark}

\begin{remark} 	
	Note that, even if $x^\star$ is shown to be $\mu$-weak through \autoref{prop:second_order_char}(iii) or (iii'), it does not preclude the possibility that $x^\star$ could in fact be strict or mere due to the looseness of the bound in the proof of \autoref{prop:second_order_char}. Furthermore, unlike strict VSS, mere VSS need not be isolated. 
	The next examples illustrate various notions of variational stability. 
\end{remark} 

\begin{example}(\textbf{Monotone game with a mere VSS})
	\label{example:two_player_zs_game} 
	Every NE of a merely monotone game is globally mere VSS. Perhaps the simplest example of a merely monotone game is the so-called bilinear saddle point problem, whereby, suppose we have a saddle function, $f(x^1, x^2) = x^1x^2$, $x^p \in \mathbb{R}$ for which we want to minimize in $x^1$ and maximize in $x^2$. We can cast this saddle point problem as a game by utilizing two payoff functions $\mathcal{U}^1(x^1; x^2) = -x^1x^2$ and $\mathcal{U}^2(x^2; x^1) = x^1x^2$. This game has a pseudo-gradient of $U(x) = (-x^2, x^1)$ and a game Jacobian $\gJ_U(x) = \mathpzc{O}$. By \autoref{prop:second_order_char}(ii), the unique interior NE $x^\star = (0,0)$ is the globally mere VSS. 
\end{example}

\begin{example}(\textbf{Non-monotone potential game with a mere VSS})
	This example shows that the mere monotonicity enjoyed by \autoref{example:two_player_zs_game} can be destroyed by the addition of another player, even when all the player's payoff functions remain linear in its own argument. 
	Considered a three-player game, whereby, 	\begin{equation}
		\textstyle \mathcal{U}^p(x^p; x^{-p}) = -x^p \prod_{q \neq p} x^{q}, \quad p \in \{1, 2, 3\},
	\end{equation} and each $x^p \in [-1, 1]$, i.e., $\Omega = [-1, 1]^3$. For this game, the pseudo-gradient is $U(x) = -(x^2x^3, x^1x^3, x^1x^2)$ which shows that there exists an interior NE at $x^\star = (0,0,0)$. The Jacobian, which is symmetric, is,
	\[\Jacob_U(x) = -\begin{bmatrix} 0&  x^3 &x^2 \\ x^3 & 0&  x^1 \\ x^2&  x^1 & 0 \end{bmatrix},
	\]
	hence this game is non-monotone in general and in fact possesses a non-convex and non-concave \textit{potential function} $P(x) = -x^1x^2x^3$, i.e., $\nabla P = U$. Since $\Jacob_U(x^\star) = \mathpzc{O}$ is negative semi-definite, hence by \autoref{prop:second_order_char}(ii'), $x^\star$ is locally merely VS. Note that if we restrict the strategy set to $\Omega = [0,1]^3$ instead, every NE $x^\star \in \{x \in \Omega| x^p = x^q= 0, p, q \in \mathcal{N}, p \neq q\}$ is a globally mere VSS, as $U(x)^\top (x - x^\star) = -3x^1x^2x^3 \leq 0, \forall x \in [0,1]^3$. 
\end{example} 

\begin{example}(\textbf{RPS with non-negative payoff for ties})
	\label{example:rps_game}
Consider a two-player Rock-Paper-Scissors (RPS) game with $\mathpzc{A}$ and $\mathpzc{B}$ being the payoff matrices for player $1$ and $2$ respectively, 
\begin{equation}  \mathpzc{A} = \begin{bmatrix} \varsigma & -\mathpzc{l} &\mathpzc{w} \\ \mathpzc{w} & \varsigma & -\mathpzc{l} \\ -\mathpzc{l} & \mathpzc{w} & \varsigma \end{bmatrix}, \mathpzc{B} = \mathpzc{A}^\top, \end{equation} 		
where $\mathpzc{l}, \mathpzc{w} \geq 0$ are the values associated with a loss or a win and $\varsigma \in \mathbb{R}$ is the payoff of a tie. The strategy set associated with this game is the simplex $\Omega^p = \{ x^p \in \mathbb{R}^{n_p}| \sum_{i = 1}^{n_p} x^p_i = 1, x^p_i \geq 0\}$. The pseudo-gradient (as known as  payoff vector \cite{Hofbauer_Stable, Sandholm}) and the game Jacobian are as follows,  \begin{equation} \label{example:rps_pseudo_grad_jacobian} U(x) = \begin{bmatrix} \mathpzc{O} &  \mathpzc{A} \\ \mathpzc{A} & \mathpzc{O} \end{bmatrix} \begin{bmatrix}   x^1 \\ x^2 \end{bmatrix} \quad  \gJ_U(x) =  \dfrac{1}{2}  \begin{bmatrix} \mathpzc{O} &  \mathpzc{A} +   \mathpzc{A}^\top \\ \mathpzc{A} +  \mathpzc{A}^\top& \mathpzc{O} \end{bmatrix}\!. \end{equation}  

To simplify our analysis, consider an example where $l = 0$. In this game, $x^\star = ({x^p}^\star)_{p \in \mathcal{N}}, {x^p}^\star = (1/3, 1/3, 1/3)$ is the unique interior NE for any $\varsigma \in [0, w)$, $\mathpzc{w} > 0$. The eigenvalues of $\gJ_U(x)$ on $T_\Omega(x)$ are $\{-(\varsigma + \mathpzc{w}), \pm (\varsigma - \mathpzc{w}/2), \pm (\varsigma - \mathpzc{w}/2)\}$, hence this game is $\mu$-weakly merely monotone with $\mu = \varsigma-\mathpzc{w}/2$ whenever $\varsigma \in (\mathpzc{w}/2,  \mathpzc{w})$. When $\varsigma = \mathpzc{w}/2$,  the game is merely monotone (\autoref{def:monotone_games}). Taken together, this means for $\varsigma \in (\mathpzc{w}/2,  \mathpzc{w})$, $x^\star$ is a unique, locally $\mu$-weak VSS where $\mu  = \varsigma- \mathpzc{w}/2$, whereas for $\varsigma = \mathpzc{w}/2$, $x^\star$ is a unique, globally mere VSS. Finally, note that when $\mathpzc{l} = \mathpzc{w} = \varsigma = 0$ (and hence both $\mathpzc{A}, \mathpzc{B}$ are zero matrices), every strategy is globally merely VS and forms a convex set (the strategy set itself).
\end{example} 

\section{Review of First-Order MD Dynamics}
\label{section:first_order_MD} 
We now describe the dynamic process that a group of agents utilizes in order to arrive at one of the equilibrium notions discussed in previous sections. One such general model is the mirror descent (MD) dynamics. Intuitively, the family of MD dynamics ascribes an abstract behavior model to each player, which states that, upon receiving the partial-gradient of its payoff function, each player processes the partial-gradient information (typically via an aggregation), then converts the processed information into the next strategy. This process can be described by the following set of ODEs  \cite{Staudigl17, Mertikopoulos16, Mertikopoulos18, Flokas},
\begin{equation}
	\vspace{-0.1cm}
	\dot z^p =   \gamma \nabla_{x^p} \mathcal{U}^p(x^p; x^{-p}) = \gamma U^p(x), \,\,\, x^p = C^p_\epsilon(z^p),
	\label{eqn:MD_p}
	\tag*{(MD)}
\end{equation}
where $z^p$ is referred to as score or dual variable, $\gamma > 0$ is a rate parameter, and $C^p_{\epsilon}: \mathbb{R}^{n_p}  \to \Omega^p$  is referred to as the \textit{mirror map}, 
\begin{equation} 
	\label{eqn:mirror_map_argmax_char} 
	C^p_{\epsilon}(z^p) = \text{argmax}_{y^p \in \Omega^p}\left[ {y^p}^\top z^p  - \epsilon \vartheta^p(y^p)\right], \epsilon > 0.
\end{equation} 
Here, $\vartheta^p: \mathbb{R}^{n_p} \to \mathbb{R}\cup\{\infty\}$ is assumed to be a closed, proper and (at least) strictly convex function, referred to as a \textit{regularizer}, where $\dom(\vartheta^p)$ is assumed be a non-empty, closed and convex set, which agrees with the strategy set $\Omega^p$, $\epsilon > 0$ is referred to as the regularization constant. The regularizer often satisfies one of the following distinct assumptions that ensure the existence of a unique solution of MD and \eqref{eqn:mirror_map_argmax_char}:
\begin{assumption}
	\label{assump:regularizer}
	$\vartheta^p\!:\! \mathbb{R}^{n_p} \!\!\to \!\mathbb{R}\!\cup\!\{\!\infty \!\}$ is closed, proper, convex, with $\dom(\vartheta^p) = \Omega^p$ non-empty, closed and convex. In addition, $\vartheta^p$ is, 
		(i) Legendre (i.e., strictly convex, steep, $\interior(\dom(\vartheta^p))\! \neq \!\emptyset$) and supercoercive, or (ii)  $\rho$-strongly convex, $\rho > 0$.
\end{assumption}
We note that for $\vartheta^p$ to be \textit{steep}, it means that $\|\nabla \vartheta^p(x_k^p)\|_\star \to +\infty$, $\|\cdot\|_\star$ is the dual norm, whenever $\{x_k^p\}_{k = 1}^\infty$ is a sequence in $\rinterior(\dom(\vartheta^p)) = \rinterior(\Omega^p)$ converging to a point in the (relative) boundary. Let ${\psi^p_\epsilon}^\star$ be the convex conjugate of ${\psi^p_\epsilon}$, $\psi^p_\epsilon \coloneqq \epsilon \vartheta^p$. Then by \autoref{lem:mirror_map_properties} and \autoref{lem:mirror_map_properties_legendre} (in the Appendix), under \autoref{assump:regularizer}, $C^p_\epsilon = \nabla {\psi^p_\epsilon}^\star$. When $\vartheta^p$ is steep, $C^p_\epsilon$ maps from $\mathbb{R}^{n_p}$ to all values in $\Omega^p$ except those at the boundary; $C^p_\epsilon$ could map to boundary points otherwise. We refer to $C^p_\epsilon$ as the \textit{mirror map induced by $\psi^p_\epsilon$} and \textit{Legendre} or \textit{strongly-induced} (by $\psi^p_\epsilon$) when $\vartheta^p$ satisfies Assumption 2(i) or (ii). Note that the notions of Legendre and strongly convex need not be dichotomous. It is possible for a $C^p_\epsilon$ to be both Legendre and strongly-induced. The following examples capture the three types of mirror maps that are extensively used in the literature \cite{Raginsky12, Krichene15, Coucheney, Staudigl17, Mertikopoulos16, Flokas, Mertikopoulos18, Bo_LP_TAC2020, Bo_Pavel_TechNote_2021, Zhou_Lossy_18}.
\begin{example}
	\label{ex:unconstrained}
	Let $\Omega^p = \mathbb{R}^{n_p}$, $\vartheta^p(x^p) = 	\frac{1}{2 \epsilon}\|x^p\|_2^2$, hence it satisfies \autoref{assump:regularizer}(i), (ii). In this case, $C^p_\epsilon = Id$ and it is both Legendre and strongly-induced. 
\end{example}

\begin{example}
	\label{exmp:softmax}
	Let $\Omega^p  = \{x^p \in \mathbb{R}^{n_p}| \sum_{i = 1}^{n_p} x_i^p = 1, x_i^p \geq 0\}$ (the unit simplex), $\vartheta^p(x^p) = \sum_{i = 1}^{n_p} {x^p_i}^\top \log(x_i^p)$, where we assume $0\log(0) = 0$. $\vartheta^p$ can be shown to be $1$-strongly convex in $\|\cdot\|_2$ (or $\|\cdot\|_1$), steep, hence it satisfies \autoref{assump:regularizer}(ii).  $\textstyle C^p_\epsilon(z^p) = (\exp(\epsilon^{-1} z^p_i) (\sum_{j = 1}^{n_p} \exp(\epsilon^{-1} z^p_j))^{-1})_{i\in[n_p]}$ is the \textit{softmax} function.
\end{example}

\begin{example}
	Let $\Omega^p \subset \mathbb{R}^{n_p}$ be a non-empty, convex, compact set,  assume $\vartheta^p(x^p) = \frac{1}{2}\|x^p\|_2^2$, hence $\vartheta^p$ satisfies \autoref{assump:regularizer}(ii) and is non-steep. Here,  $C_\epsilon^p = \pi_{\Omega^p}$ is the Euclidean projection onto $\Omega^p$. 
\end{example}

We can represent MD in a more compact stacked notation, 
\begin{equation}
	\label{eqn:MD} 
	\dot z  =   \gamma U(x), \quad x = C_\epsilon(z), 
\end{equation}
where $x = (x^p)_{p \in \mathcal{N}}, z = (z^p)_{p \in \mathcal{N}}, C_\epsilon = (C^p_\epsilon)_{p \in \mathcal{N}}, U = (U^p)_{p \in \mathcal{N}}$.	Observe that at rest, MD \eqref{eqn:MD} satisfies, \begin{equation} \label{eqn:MD_rest} \mathbf{0} =  U(\overline x), \quad \overline x = C_\epsilon(\overline z), \quad \overline z \in C_\epsilon^{-1}(\overline x), \end{equation} where $\overline x = ({\overline x^p})_{p \in \mathcal{N}}, \overline z = ({\overline z^p})_{p \in \mathcal{N}}$ are the rest points of MD and $C_\epsilon^{-1} = ((C^p_\epsilon)^{-1})_{p \in \mathcal{N}} = (\partial \psi^p_\epsilon)_{p \in \mathcal{N}}$ is the inverse (or the pre-image) of $C_\epsilon$. The rest condition \eqref{eqn:MD_rest} implies that $\overline x$ is an interior NE. Hence if a trajectory $z(t)$ of MD comes to a rest, $x(t) = C_\epsilon(z(t))$ converges to an interior NE.  We note that the uniqueness of $\overline x$ does not imply the uniqueness of $\overline z$ unless $C_\epsilon$ is Legendre-induced (for which $C_\epsilon$  is one-to-one; this follows from Legendre theorem \cite{Bauschke97}). Hence, in general, the convergence of MD is to a set and not to an equilibrium and $z(t)$ may continue to evolve even after $x(t)$ has reached an equilibrium. Note that rest points are not the only game relevant solutions for which MD may converge to. As pointed out in \cite{Mertikopoulos16}, there are non-rest points that occur on the boundary of $\Omega$ that are asymptotically stable under MD. The following lemma broadly summarizes the key convergence properties of MD (refer to \cite{Staudigl17, Mertikopoulos16} for proofs).

\begin{lem}
	\label{lem:MD}	
	Let $\mathcal{G}$ be a concave game with a globally strict VSS $x^\star$. Suppose that all players choose strategies according to MD \eqref{eqn:MD}. Let $x(t) = (x^p(t))_{p \in \mathcal{N}} = C_\epsilon(z(t))$ be generated by \eqref{eqn:MD} and $C_\epsilon = (C_\epsilon^p)_{p \in \mathcal{N}}$ be induced by $\psi^p_\epsilon \coloneqq \epsilon \vartheta^p$. For any $\gamma, \epsilon > 0$ and any $x(0) = C_\epsilon(z(0)) \in \Omega, z(0) \in \mathbb{R}^n$, 
	\begin{itemize}
		\item[(i)] suppose $\vartheta^p$ satisfies \autoref{assump:regularizer}(i), $\forall p$, then $x(t) = C_\epsilon(z(t))$ converges to $x^\star$, whenever $x^\star$ is  interior.
		\item[(ii)] suppose $\Omega^p$ is compact and $\vartheta^p$ satisfies \autoref{assump:regularizer}(ii), $\forall p$, then $x(t) = C_\epsilon(z(t))$ converges to $x^\star$. 
	\end{itemize}
\end{lem}

In general, MD does not converge beyond strict VSS, i.e., a mere VSS. This poses considerable limitations in practice, for instance, mere VSS are commonly found in ZS games. \cite{Hofbauer_Stable, Sandholm} showed that every ZS finite game is merely (but not strictly) monotone, hence all of its NEs are mere (non-strict) VSS.  A standard method to overcome non-convergence to the NE in ZS games is to calculate a time-averaged (\textit{ergodic}) trajectory  $x_\text{avg}(t) = t^{-1} \int\nolimits_{0}^t x(\tau) \mathrm{d}\tau$ in tandem with MD, that is,
\begin{equation}
	\label{eqn:MD_avg} 
	\dot z  =   \gamma U(x), \quad x = C_\epsilon(z),  \quad x_\text{avg}(t) = t^{-1} \int\nolimits_{0}^t x(\tau) \mathrm{d}\tau,
	\tag*{(MDA)}
\end{equation}
for which the time-averaged strategy $x_\text{avg}$ has been shown in many contexts to converge, e.g., \cite{Mertikopoulos16}. The main critique of using MDA is that the actual strategies do not arrive at the NE in the long-run, thereby making it unsuitable for on-line equilibrium seeking in the absence of a central planner or coordination between players. Furthermore, averaging may fail to converge outside of ZS games \cite{Unstable_Equilibria}, which makes this approach vulnerable to parameter perturbation.  

Another method for overcoming non-convergence is through \textit{discounting}, which was studied in \cite{Bo_Pavel_TechNote_2021}, 
\begin{equation}
	\label{eqn:DMD} 
	\dot z =   \gamma (-z + U(x)), \quad x = C_\epsilon(z), 
	\tag*{(DMD)}
\end{equation}
where DMD stands for \textit{discounted} mirror descent. Compared to MD, an extra $-z$ term is inserted in the $\dot z$ system, which translates into an exponential weighted decay (or discounting) term in the closed-form solution $z(t)$. DMD has a connection with the so-called \textit{weight decay} method in the machine learning literature \cite{Krogh92}, as $z \in C_{\epsilon}^{-1}(x)$ can be shown to be equivalent to a (usually non-Euclidean) regularization term, which directly interacts with the monotonicity property of $U$. While it is known that DMD could converge exactly in subclasses of finite games that exhibit symmetric interior NEs, whereby $C_\epsilon$ is also chosen to enforce symmetry (see \cite{Bo_LP_TAC2020}), in general, it cannot converge exactly to a NE, which also means that it cannot converge exactly to a VSS. This directs our attention to alternatives methods, such as higher-order augmentation of game dynamics \cite{Bo_LP_TAC2020, Laraki13, Shamma_anticipatory, Flam}. 

\section{Second-Order Mirror Descent Dynamics} 
\label{section:second_order_MD} 
We now propose the \textit{second-order mirror descent}, which in terms of each player $p$ appears as,
\begin{equation}
	\label{eqn:MD2_p}
	\begin{cases}
		\dot z^p &\! =  \gamma (U^p(x) \!-\! \alpha (x^p - \xi^p)),\,\, \quad 	\dot \xi^p \! = \beta (x^p \!-\! \xi^p),\\
		x^p &\! =  C^p_\epsilon(z^p),
	\end{cases}
	\tag*{(MD2)}
\end{equation} 
and in stacked notation,
\begin{equation}
	\label{eqn:MD2}
		\dot z  =  \gamma (U(x) \! - \!\alpha (x \!-\! \xi)),\,\,\quad \dot \xi = \beta (x - \xi),  \quad x =  C_\epsilon(z),
\end{equation} 
where  $x = (x^p)_{p \in \mathcal{N}}$, $z = (z^p)_{p \in \mathcal{N}}$, $\xi = (\xi^p)_{p \in \mathcal{N}}$, $C_\epsilon = (C^p_\epsilon)_{p \in \mathcal{N}}$, $U = (U^p)_{p \in \mathcal{N}}$. The rest point condition for MD2 is,  
\begin{equation}
	\label{eqn:MD2_rest_condition}
	\mathbf{0} =   U(\overline x),  \quad \mathbf{0}  =  \overline x - \overline \xi, \quad \overline x =  C_\epsilon(\overline z), \quad \overline z  \in  C^{-1}_\epsilon(\overline x),
\end{equation} 
which coincides with that of MD, i.e., $\overline x = x^\star$ is an interior NE. 

\begin{remark}(\textbf{Intuitive Learning Interpretation of MD2})
	We can explicitly write  $\xi^p(t) = (\xi^p_i(t))_{i \in [n_p]}$ as,  \begin{equation}
		\label{eqn:individual_xi_p_expression}
		\xi^p_i(t) = e^{{-\beta}t} \xi_i^p(0) + \beta \smallint\nolimits_0^t e^{{-\beta}(t-\tau)} x^p_i(\tau) \mathrm{d}\tau,
	\end{equation}
	which represents an ``exponential weighting" of the strategies. Hence, we refer to $\xi$ as the \textit{primal aggregate}, and $z$ can be referred as the \textit{dual aggregate}. We note that $\xi^p$ resides in the unconstrained (ambient) space $\mathbb{R}^{n_p}$ containing the action set. Note that, whenever $\xi^p(0)$ is initialized in $\Omega^p$, then $\xi^p(t) \in \Omega^p$ for all times. Furthermore, noting that we can re-write $\dot z^p$ subsystem in MD2 as $\dot z^p \! =  \gamma (U^p(x) \!-\! \alpha \beta^{-1} \dot  \xi^p)$. By assuming $z^p_i(0)  = 0, \forall i, p$,
	\begin{align*}
		& z^p_i(t) = \gamma (\smallint\nolimits_0^t U^p_i(x(\tau)) \mathrm{d}\tau -   \alpha \beta^{-1}\smallint\nolimits_{0}^t \dot \xi^p_i(\tau) \mathrm{d}\tau) \numberthis \label{eqn:individual_z_p_expression}\\
		& = \gamma (\smallint\nolimits_0^t U^p_i(x) \mathrm{d}\tau \!-\!   \alpha e^{{-\beta}t}(\beta^{-1}(1\!-\!e^{{\beta}t}) \xi_i^p(0) \!+\! \smallint\nolimits_0^t e^{\beta\tau} x^p_i\mathrm{d}\tau)).
	\end{align*}
	It can be seen, when $t$ is small, the effects of $\xi^p_i(0)$ and $x^p_i$ are at their largest, so the presence of $\dot \xi$ term takes into consideration the uncertainty during the initial periods of play. Hence, the incorporation of  $\dot \xi$ has an \textit{exploratory} effect on the play, which can result in strategies being played more conservatively, preventing the players from immediately reaching a deadlock. 
\end{remark}
\begin{remark}(\textbf{Comparison with Existing Dynamics})
	The closest dynamics related to MD2 is the higher-order exponentially discounted dynamics (H-EXPD-RL) for finite games \cite{Bo_LP_TAC2020}. 	Specifically, when the higher-order augmentation is taken to be a high-pass filter, we obtain,
	\begin{equation}
		\hspace*{-0.5cm}
		\begin{cases}
			\dot z & =  \gamma (U(x) - z - \alpha (x + \xi)),\,\, \quad \dot \xi = \beta (-x - \xi),\\
			x &=  C_\epsilon(z).
		\end{cases}
		\label{eqn:H-EXP-D-RL}
	\end{equation} 
	\eqref{eqn:H-EXP-D-RL} can be seen as the second-order extension of DMD. In \cite{Bo_LP_TAC2020} it was observed that \eqref{eqn:H-EXP-D-RL} is robust to a greater degree of parameter perturbation in monotone games as compared to DMD.
	
	Let's now consider MD2 as a purely second-order ODE in the dual-space. To simplify our calculations, we assume $x(t) = C_\epsilon(z(t))$ is differentiable. By using MD2 and noting that $\dot z \! =  \gamma (U(x) \!-\! \alpha\beta^{-1} \dot  \xi))$, taking a time derivative of $\dot z$ and making the assumption that $U$ and $C_\epsilon$  are $\mathcal{C}^1$, we have, $
	\ddot z  =  \gamma (\J_{U}(x)\dot x  - \alpha\beta^{-1} \ddot \xi)$, 
	which, along with $\ddot \xi = \beta (\dot x - \dot \xi)$, $x = C_\epsilon(z)$,  $\dot x = \J_{C_\epsilon(z)}\dot z$, $\dot \xi = \beta \alpha^{-1} (U(x) - \gamma^{-1} \dot z)$, we obtain,
	\begin{equation} \hspace{-0.2cm}  
	\ddot z  = \gamma \left[\J_{U}(C_\epsilon(z))\J_{C_\epsilon}(z) - \alpha \J_{C_\epsilon}(z) - \beta\gamma^{-1}\mathpzc{I}\right] \dot z + \gamma \beta U(x),
	\end{equation} 
	or by using the chain-rule for Jacobian, 
	\begin{align}	
		\begin{cases} 
			\ddot z & = \gamma \left[\J_{U \circ C_\epsilon}(z) - \alpha \J_{C_\epsilon}(z) - \beta\gamma^{-1} \mathpzc{I}\right] \dot z + \gamma\beta U(x),\\
			x & = C_\epsilon(z),
		\end{cases} 
	\end{align}
	where $U \circ C_\epsilon \coloneqq U(C_\epsilon)$. To the best of our knowledge, most of the existing second-order continuous gradient-type dynamics (see \cite{Attouch2020} and references therein) \textit{cannot} recover MD2 due to the presence of the Jacobian of $C_\epsilon$. This is because in the unconstrained, primal setting (as studied in \cite{Attouch2020, Dian_LP_CDC2020}), $C_\epsilon$ is the identity and hence, $\J_{C_\epsilon}(z) = \I, \forall z$ and its contribution is lumped together with $\beta {\gamma}^{-1}\I$. 
\end{remark}

We proceed to demonstrate that the primal aggregate $\xi$ contributes to the convergence of MD2 beyond strict VSS and such convergence depends on both the property of the regularizer as well as the topological properties of the underlying strategy set. 

\begin{thm}
	\label{thm:main}
	Let $\mathcal{G}$ be a concave game and assume that every interior NE is globally merely VS. Suppose that all players choose strategies according to MD2 \eqref{eqn:MD2}. Let $x = (x^p(t))_{p \in \mathcal{N}} = C_\epsilon(z(t))$ be generated by \eqref{eqn:MD2} and $C_\epsilon = (C_\epsilon^p)_{p \in \mathcal{N}}$ be induced by $\psi^p_\epsilon \coloneqq \epsilon \vartheta^p$. Let $x^\star$ denote an interior mere VSS (possibly from a set of them). For any $\alpha, \beta, \gamma, \epsilon > 0$,  $x(0) = C_\epsilon(z(0)) \in \Omega, z(0), \xi(0) \in \mathbb{R}^n$,
	\begin{itemize}
	\item[(i)] suppose $\vartheta^p$ is twice-continuously differentiable ($\mathcal{C}^2$) and satisfies \autoref{assump:regularizer}(i), $\forall p$, then $x(t) = C_\epsilon(z(t))$ converges to $x^\star$. 
	\item[(ii)] assume $\Omega^p$ is compact and $\vartheta^p$ satisfies \autoref{assump:regularizer}(ii), $\forall p$, then $x(t) = C_\epsilon(z(t))$ converges to $x^\star$. 
	\end{itemize}
\end{thm}

We now sharpen \autoref{thm:main} by considering several other classes of games in which MD2 is guaranteed to converge. From \cite{MigotCojocaru2020, Facchinei_I},
\begin{definition}
	The game $\mathcal{G}$ is, 
	\begin{enumerate}[noitemsep,topsep=0pt]
		\item[(i)]  pseudo-monotone (resp. strictly pseudo-monotone), if $U(x^\prime)^\top(x - x^\prime) \leq 0 \implies U(x)^\top (x - x^\prime) \leq 0, \forall x, x^\prime \in \Omega$ (resp. $U(x)^\top (x - x^\prime) < 0$, with equality iff $x = x^\prime$.)
		\item[(ii)] quasi-monotone (resp. strictly quasi-monotone), $U(x^\prime)^\top(x - x^\prime) < 0 \implies U(x)^\top (x - x^\prime) \leq 0, \forall x, x^\prime \in \Omega$ (resp. $U(x)^\top (x - x^\prime) < 0$, with equality iff $x = x^\prime$).
	\end{enumerate}
\end{definition}
Observe that when $x^\prime$ in the above definitions is NE $x^\star$ of a pseudo-monotone game (or a globally \textit{strict NE} of a quasi-monotone game), it automatically implies the associated NE is globally mere VS. Moreover, $x^\star$ is globally strictly VS when the above games are strict. Since MD converges to globally strict VSS \cite{MZ2019}, therefore it converges to NE in strictly pseudo-monotone games and strict NE in strictly quasi-monotone games. Our next corollary, which immediately follows from \autoref{thm:main}, partially generalizes these results to the non-strict setting under MD2. 

\begin{cor}
	\label{cor:convergence_to_NE}
	Let $\mathcal{G}$ be a concave game and assume that every NE is interior. Suppose that all players choose strategies according to MD2 \eqref{eqn:MD2}. Let $x = (x^p(t))_{p \in \mathcal{N}} = C_\epsilon(z(t))$ be generated by \eqref{eqn:MD2} and $C_\epsilon = (C_\epsilon^p)_{p \in \mathcal{N}}$ be induced by $\psi^p_\epsilon \coloneqq \epsilon \vartheta^p$, $\vartheta^p$ is $\mathcal{C}^2$ and satisfies \autoref{assump:regularizer}(i), $\forall p$, then for any $\alpha, \beta, \gamma, \epsilon > 0$ and any $x(0) = C_\epsilon(z(0))$, 
	\begin{itemize}
		\item[(i)] $x(t)$ converges to an interior NE $x^\star$ whenever  $\mathcal{G}$ is merely monotone or pseudo-monotone, and, 
		\item[(ii)] $x(t)$ converges to an interior strict NE $x^\star$ if $\mathcal{G}$ is quasi-monotone, whenever $x^\star$ exists. 
	\end{itemize} 
	The same conclusions hold whenever $\vartheta^p$ satisfies \autoref{assump:regularizer}(ii) and $\Omega^p$ is compact $\forall p$. 
\end{cor}

\begin{remark}
	We note that the existence of a NE in pseudo-monotone games is guaranteed under our continuous game assumption \cite[Theorem 2.3.5]{Facchinei_I}, possibly requiring $-\mathcal{U}^p$ to be coercive, see \autoref{assump:1}, and the non-empty set of NEs coincides with the set of globally mere VSS \cite[Theorem 2]{MigotCojocaru2020}.  However, \cite{MigotCojocaru2020} pointed out that the existence of a globally merely VS for quasi-monotone games is not guaranteed in general. Instead, a globally merely VS exists under the stronger property of \textit{properly} quasi-monotone \cite[Theorem 3]{MigotCojocaru2020}.
\end{remark}

\section{Additional Convergence Properties of MD2}
\label{section:additional_properties_MD2} 

In this section, we investigate two additional properties of MD2, namely, that of rate of convergence and regret minimization. To account for the geometry of the problem, we provide a \textit{relative} extension to the strong VSS.
\begin{definition}	\label{def:relatively_strongly_vs}
	Let $h: \mathbb{R}^n \to \mathbb{R} \cup \{\infty\}$ be any differentiable, convex function with domain $\dom(h) = \Omega$.  Then $x^\star \in \Omega$ is $\eta$-relatively strongly VS (with respect to $h$) if for all $x \in \Omega$, $U(x)^\top(x-x^\star) \leq -\eta(D_h(x, x^\star) + D_h(x^\star, x)),$  for some $\eta \!>\! 0$, where $D_h$ is the Bregman divergence of $h$.
\end{definition}
We note that \autoref{def:relatively_strongly_vs} is analogous to that of a \textit{relatively strongly monotone} game, which was previously introduced in \cite{Bo_Pavel_SIAM_2022}. Following \cite[Theorem 4.4]{Bo_Pavel_SIAM_2022}, it can be shown that MD converges to a relatively strongly VSS in $\mathcal{O}(e^{-\gamma \eta \epsilon^{-1} t})$ given a strongly-induced $C_\epsilon$ that is adapted to the geometry of this VSS. We now wish to provide a similar result for the convergence of MD2 towards a relatively strong VSS. However, exponential convergence does not follow from \autoref{thm:main}. Instead, we propose an augmented version of MD2 that exhibits exponential convergence,
\begin{equation}
	\label{eqn:MD2_gamma} 
	\begin{cases}
		\dot z & =   (U(x)  - \gamma\alpha (x - \xi)), \qquad \dot \xi = \beta (x - \xi),\\
		x &=  C_\epsilon(z), \qquad \qquad \qquad \qquad \dot \gamma  = -\eta \epsilon^{-1} \gamma, 
	\end{cases}
	\tag*{(MD2$\gamma$)}
\end{equation}
where $\gamma(0) > 0$. Observe that MD2$_\gamma$ is equivalent to the non-autonomous system,
\begin{equation}
	\hspace*{-0.3cm}
	\label{eqn:MD2_gamma_perturbed} 
	\begin{cases}
		\dot z & =  (U(x)  -  e^{-\eta \epsilon^{-1} t} \gamma(0)\alpha (x - \xi)), \quad \dot \xi  = \beta (x - \xi),\\
		x &=  C_\epsilon(z),
	\end{cases}
\end{equation}
whose rest point condition coincides with that of MD2. 
\begin{remark}
	From \eqref{eqn:MD2_gamma_perturbed}, MD2$\gamma$ can be seen as MD with a \textit{vanishing perturbation} $g(t, z, \xi) = e^{-\eta \epsilon^{-1} t} \alpha(0) (x - \xi)$, which allows for the following simple learning interpretation: when the players are aware that the game being played has a $\eta$-strong VSS, they no longer bother with exploring the strategy space during the initial stages and instead discard the extra information represented by $x - \xi$ exponentially fast.    
\end{remark}

\begin{thm}
	\label{thm:MD2_gamma}
	Let $\mathcal{G}$ be a concave game with a unique interior $\eta$-strongly VSS  relative to $h(x) = \sum_{p \in \mathcal{N}} \vartheta^p(x^p)$, denoted by $x^\star$,  where $x = (x^p(t))_{p \in \mathcal{N}} \!=\! C_\epsilon(z(t))$ is generated by MD2$\gamma$ and $C_\epsilon \!=\! (C_\epsilon^p)_{p \in \mathcal{N}}$ is induced by $\psi^p_\epsilon\! \coloneqq  \!\epsilon \vartheta^p$, $\vartheta^p$ satisfies \autoref{assump:regularizer}(ii). Then for any $\alpha, \beta, \epsilon, \eta, \gamma_0\! =\! \gamma(0) \!> \!0$ and any $x_0\! =\! (x^p(0))_{p \in \mathcal{N}} \!=\! C_\epsilon(z(0)), z(0), \xi_0 \!=\!\xi(0) \in \mathbb{R}^n$, $x$ converges to $x^\star$ with the rate,
	\begin{equation}
		\label{eqn:MD2_rate}
		D_h(x^\star, x) \leq e^{-  \eta{\epsilon}^{-1} t}(\dfrac{\alpha\gamma_0}{2\beta \epsilon} \|\xi_0 - x^\star\|_2^2 +  D_h(x^\star, x_0)).
	\end{equation}  
	Furthermore, since $\vartheta^p$ satisfies \autoref{assump:regularizer}(ii), i.e., $\vartheta^p$ is $\rho$-strongly convex, therefore,
	\begin{equation}
		\hspace{-0.2cm}
		\label{eqn:MD2_rate_sc}
		\|x^\star-  x\|_2^2 \leq 2\rho^{-1} e^{- \eta{\epsilon}^{-1} t}(\dfrac{\alpha\gamma_0}{2\beta\epsilon} \|\xi_0 - x^\star\|_2^2  +  D_h(x^\star, x_0)).
	\end{equation}  
\end{thm}

\begin{remark}
	\label{remark:rate_result} 
	The above results confirm our intuition: the upper-bound rate of convergence diminishes faster when an equilibrium is more strongly VSS $(\eta \uparrow)$, when $C_\epsilon$ is induced by an even more strongly convex regularizer ($\rho \uparrow$), when the regularization parameter goes down ($\epsilon \downarrow$), so that $C_\epsilon$ approximates a \textit{best-response map}, or when the extra information $x - \xi$ diminishes more rapidly ($\beta \uparrow$). 
\end{remark}

Next, we turn our attention towards the question of regret minimization. Let us define the \textit{time-averaged (external/static) regret function} of player $p \in \mathcal{N}$ as, $\mathpzc{R}^p: [0, \infty) \to \mathbb{R}$, $\mathpzc{R}^p(0) = 0$,
\begin{equation} \label{eqn:regret}  \mathpzc{R}^p(t)  = \max\limits_{y^p \in \Omega^p} \dfrac{1}{t} \smallint\nolimits_{0}^{t} \mathcal{U}^p(y^p; x^{-p}(\tau)) - \mathcal{U}^p(x(\tau))  \mathrm{d}\tau. \end{equation} 

Intuitively, the regret value at time $t$ represents the time-averaged sum total of the payoff difference between the actual strategy $x(t) = (x^p(t); x^{-p}(t))$ versus the best strategy $(y^p; x^{-p}(t))$ that the player $p$ could have played in hindsight, given any opponents' strategy $x^{-p}(t) \in \Omega^{-p}$.  Player $p$'s dynamics that generate $x^p(t)$ is said to achieve \textit{no-regret} with respect to \eqref{eqn:regret} if  $\textstyle \limsup_{t \to \infty} \mathpzc{R}^p(t) \leq 0$.   

While MD was shown to achieve no-regret in finite games, it cannot converge in ZS finite games with an interior mixed equilibrium \cite{Mertikopoulos18}. In contrast to \cite{Mertikopoulos18, Flokas}, our next result, coupled with \autoref{thm:main}, shows that MD2 achieves both no-regret as well as exact convergence whenever the (interior) equilibrium is a mere VSS.

\begin{thm} \label{thm:no_regret} Let $\mathcal{G}$ be a concave game, where $\Omega^p$ is compact for all $p \in \mathcal{N}$. Then for every continuous trajectory of $x^{-p}(t)$ of opponents of player $p$, each player that uses MD2 achieves no-regret independently from the rest of the players. \end{thm}

\vspace{-0.2cm}

\section{Construction of Second-Order Primal-Space Dynamics}
\label{section:second_order_primal} 
In this section, we show that MD2 can either be used to create new primal-space dynamics (that is, dynamics that solely evolve on $\Omega$) or recover existing ones. To simplify our presentation and to avoid the technicality of non-differentiable mirror maps $C_\epsilon$ (for which examples of induced primal-space dynamics can still be generated through a technical treatment, see the so-called \textit{projection dynamics} discussed in \cite{Mertikopoulos16}), we impose the following general assumption,
\begin{assumption}
	\label{assump:induced_primal} 
	$\vartheta^p\!:\! \mathbb{R}^{n_p} \!\!\to \!\mathbb{R}\!\cup\!\{\!\infty \!\}$ is steep and induces a $\mathcal{C}^1$ mirror map $C_\epsilon  = (C^p_\epsilon)_{p \in \mathcal{N}} = (\nabla \psi^{p\star}_\epsilon)_{p \in \mathcal{N}}$.
\end{assumption}
Under this assumption, taking the time-derivative of $x = C_\epsilon(z)$ shows that MD2 can be written as a pair of differential inclusions,
\begin{equation}
	\label{eqn:primal_first_order_MD2}
	\dot \xi  \in \beta(x - \xi), \quad \dot x  \in \gamma \J_{C_\epsilon}(z) (U(x) - \alpha (x- \xi)),
\end{equation}
where $\gamma > 0, \alpha, \beta \geq 0$ and the Jacobian of $C_\epsilon$ is, \[\J_{C_\epsilon}(z)  \coloneqq  \text{blkdiag}(\J_{C^p_\epsilon}(z^p)) = \begin{bsmallmatrix} \J_{C^1_\epsilon}(z^1) & & \\ & \ddots &  \\ & & \J_{C^N_\epsilon}(z^N) \end{bsmallmatrix} \in \mathbb{R}^{n \times n},\] 
and each $\J_{C^p_\epsilon}(z^p)  = \J_{{\nabla \psi^p_\epsilon}^\star}(z^p)  = \nabla^2 {\psi_\epsilon^p}^\star(z^p) \in \mathbb{R}^{n_p \times n_p}$. 
We note that the inclusions in \eqref{eqn:primal_first_order_MD2} come from  $z \in C^{-1}_\epsilon(x)$. We can further express \eqref{eqn:primal_first_order_MD2} in each $p$ as, 
\begin{equation}
	\label{eqn:primal_first_order_MD2_p}
	\dot \xi^p  \in \beta(x^p - \xi^p), \,\, 
	\dot x^p  \in \gamma \nabla^2 {\psi_\epsilon^p}^\star(z^p) (U^p(x) - \alpha (x^p - \xi^p)).
\end{equation} In the following, we offer two general ways of re-writing \eqref{eqn:primal_first_order_MD2_p} as a set of ODEs in the primal-space. 

\textbf{General Case}: Our next result shows that, in the general case, finding the primal-space dynamics associated with MD2 generally amounts to evaluating the operator  $\nabla^2 {\psi^p_\epsilon}^\star \circ \nabla \psi^p_\epsilon \coloneqq \nabla^2 {\psi^p_\epsilon}^\star(\nabla \psi^p_\epsilon)$. 

\begin{prop}
	\label{prop:primal_dynamics_general}
	Suppose \autoref{assump:induced_primal} holds, then \eqref{eqn:primal_first_order_MD2_p} reduces to, \begin{equation}
		\label{eqn:primal_first_order_nonLegendre}
		\begin{cases}
			\dot \xi^p & =\beta(x^p - \xi^p)\\
			\dot x^p  &= \gamma \nabla^2 {\psi^p_\epsilon}^\star(\nabla \psi^p_\epsilon(x^p)) (U^p(x) - \alpha (x^p -  \xi^p)). 
		\end{cases} 
	\end{equation}
\end{prop}

\begin{example}(\textbf{Simplex})
	Let $\Omega^p = \{x^p \in  \mathbb{R}^{n_p}_{\geq 0}|\|x^p\|_1 = 1\}$ and consider $\vartheta^p(x^p) = \sum_{i = 1}^{n_p} x_i^p \log(x_i^p)$, $\psi_\epsilon^p(x^p) =\epsilon \vartheta^p(x^p)$ and ${\psi^p_\epsilon}^\star(z^p) = \epsilon\log(\sum_{i = 1}^{n_p} \exp(\epsilon^{-1} z^p_i))$. It can be shown that,
	$\nabla^2 {\psi_\epsilon^p}^\star(z^p) = \epsilon^{-1}(\diag(C^p_\epsilon(z^p)) - C^p_\epsilon(z^p)C^p_\epsilon(z^p)^\top)$ where $\textstyle C^p_\epsilon(z^p) = (\exp(\epsilon^{-1} z^p_i) (\sum_{j = 1}^{n_p} \exp(\epsilon^{-1} z^p_j))^{-1})_{i\in[n_p]}$ is the \textit{softmax} function (\autoref{exmp:softmax}). Let $x^p \in \rinterior(\Omega^p)$. We can show $\nabla \psi_\epsilon^p(x^p) = \epsilon(\log(x^p) + \mathbf{1})$ ($\log$ is applied component-wise)  and $ C^p_\epsilon(\nabla \psi^p_\epsilon(x^p)) = x^p$. Taken together, we have, 
	$\nabla^2 {\psi^p_\epsilon}^\star \circ \nabla \psi^p_\epsilon = \nabla^2 {\psi^p_\epsilon}^\star(\nabla \psi^p_\epsilon(x^p))= \epsilon^{-1}(\diag(x^p) - x^p {x^p}^\top)$. By applying  \eqref{eqn:primal_first_order_nonLegendre} and defining $\mathpzc{H}(x^p) \coloneqq \diag(x^p)\! - \!x^p {x^p}^\top$, we obtain \eqref{eqn:RD2} as the induced primal-space dynamics of MD2 on the simplex, 
	\begin{equation}
	\label{eqn:RD2}
	\dot \xi^p \!=\!\beta(x^p \!-\! \xi^p), \quad 	\dot x^p  \!=\! \gamma \epsilon^{-1} \mathpzc{H}(x^p) (U^p(x)\! -\! \alpha (x^p \!-\! \xi^p)),
	\end{equation} 
	We note that \eqref{eqn:RD2} represents a version of second-order \textit{replicator dynamics} (RD), which is different from the ones introduced in \cite{Laraki13, Shamma_anticipatory}. In particular, when $\alpha = \beta = 0$, we recover the (multi-population) replicator dynamics \cite{MigotCojocaru2020,  Hofbauer_Stable, Sandholm, Mertikopoulos16, Mertikopoulos18, Sorin16}. 
	
	We can also reduce \eqref{eqn:RD2} to a \textit{single population} model ($p = 1$), which takes the following form:
	\begin{equation}
		\label{eqn:RD2_single_pop}
		\textstyle 
		\begin{cases}
		\dot \xi_i  & =\beta(x_i -  \xi_i),\\
		\dot x_i & = x_i(U_i(x) - \dfrac{\alpha}{\beta} \dot \xi_i - \sum_{j \in \mathcal{A}} x_j(U_j(x) - \dfrac{\alpha}{\beta}\dot \xi_j)),
		\end{cases} 
	\end{equation} 
	where $\mathcal{A} = \{1, \ldots, n\}$ represents $n$ subpopulations of non-atomic players, and $U_i$ represents the fitness of subpopulation $i$. We note that structurally \eqref{eqn:RD2_single_pop} shares similarity with the anticipatory replicator dynamics studied in \cite{Shamma_anticipatory}, which can be written as,
	\begin{equation}
		\textstyle 
		\begin{cases}
			\dot \xi_i  & =\beta(x_i -  \xi_i),\\
			\dot x_i & = x_i(U_i(x + \alpha \dot \xi) - \sum_{j \in \mathcal{A}} x_j(U_j(x + \alpha \dot \xi))).
		\end{cases} 
	\end{equation} 
	
	 Note that the variable $\xi_j$ in \eqref{eqn:RD2_single_pop} does not directly interacts with the pseudo-gradient/payoff vector $U(x)$, as is the case for the anticipatory RD.

\end{example}

\textbf{Legendre and Supercoercive}: A more structured scenario is when $\psi^p_\epsilon \coloneqq \epsilon \vartheta^p$ is Legendre and supercoercive for all $p$ (\autoref{assump:regularizer}(i)). Since $\psi^p_\epsilon$ is finite on $\interior(\Omega^p)$, coercive, strictly convex and $\mathcal{C}^2$, then applying \cite[Example 11.9, p. 480]{Wets09}, ${\psi^p_\epsilon}^\star$ shares the same properties and we have that $[\nabla^2 \psi^p_\epsilon (x^p)]^{-1} = \nabla^2 {\psi^p_\epsilon}^\star(z^p) =  ({\nabla^2 \psi_\epsilon^p}^\star \circ \nabla \psi_\epsilon^p)(x^p), \forall x^p = \nabla {\psi^p_\epsilon}^\star(z^p) \in \interior(\Omega^p)$,  which allows us to write \eqref{eqn:primal_first_order_MD2_p} or \eqref{eqn:primal_first_order_nonLegendre} as, 
\begin{equation}
	\label{eqn:primal_first_order_Legendre}
	\begin{cases}
		\dot \xi^p & =\beta(x^p - \xi^p),\\
		\dot x^p  &= \gamma [\nabla^2 \psi^p_\epsilon (x^p)]^{-1}(U^p(x) - \alpha (x^p- \xi^p)),
	\end{cases} 
	\tag*{(NG2)}
\end{equation}
which we refer to as the \textit{second-order natural gradient descent} (NG2). NG2 is so named because in the optimization setup ($p = 1$), for $U = -\nabla f$, where $f: \mathbb{R}^n \to \mathbb{R} \cup\{\infty\}$ is some objective function and $\alpha, \beta = 0$,  is analogous to the so-called \textit{natural gradient descent} \cite{Amari16}, $\dot x = -[\nabla^2 \psi_{\epsilon}(x)]^{-1}\nabla f(x).$	

\begin{example}(\textbf{Unconstrained}) Let $\Omega^p = \mathbb{R}^{n_p}$, $\vartheta^p(x^p) = 		\frac{1}{2}\|x^p\|_2^2$, then  $\nabla^2 {\psi^p_\epsilon}(x^p) = \epsilon \mathpzc{I}$ where ${\psi^p_\epsilon} = \epsilon \vartheta^p$. By comparing with NG2, we obtain,
\begin{equation}
	\dot \xi^p  =\beta(x^p - \xi^p), \quad \dot x^p  = \gamma \epsilon^{-1} (U^p(x) - \alpha (x^p- \xi^p)),
	\label{eqn:MD2_induced_primal_unconstrained}
\end{equation}
which recovers the one recently introduced in \cite{Dian_LP_CDC2020}. \eqref{eqn:MD2_induced_primal_unconstrained} can be further shown via discretization to recover algorithms such as optimistic gradient-descent/ascent, Polyak's heavy-ball method, among others (see \cite{Dian_LP_CDC2020}). While \eqref{eqn:MD2_induced_primal_unconstrained} is by far the most standard choice for unconstrained action sets, more than one type of dynamics can reside on the same action sets. To witness, let $\Omega^p = \mathbb{R}^{n_p}$,  $\vartheta^p(x^p) = \sum_{i = 1}^{n_p} x_i^p \arcsinh(x^p_i/\wp) - \sqrt{x_i^{p2} + \wp^2}, \wp > 0$, which represents an interpolation between entropic and Euclidean terms \cite{Ghai}. $\vartheta^p$ is Legendre and supercoercive (its convex conjugate is $\sum_{i = 1}^{n_p} \wp\cosh(x^p_i)$, which is also Legendre). Let $\psi^p_\epsilon = \epsilon \vartheta^p$, then $\nabla \psi^p_\epsilon = \epsilon \sum_{i = 1}^{n_p} \arcsinh(x_i^p/\wp)$ and  $\nabla^2 \psi^p_\epsilon  =  \epsilon \left[\diag(\begin{bmatrix} \sqrt{{x^p_1}^2 + \wp^2}, \ldots, \sqrt{x^{p2}_{n^p} + \wp^2}) \end{bmatrix}^\top)\right]^{-1}.$ By comparing with NG2, we obtain, 
	\begin{equation}
	\dot \xi^p  =\beta(x^p - \xi^p), \quad \dot x^p  = \gamma \epsilon^{-1} \mathpzc{S}(x^p) (U^p(x) - \alpha (x^p- \xi^p)). 
	\label{eqn:MD2_induced_primal_unconstrained_hyentropy}
	\end{equation}	
where $\mathpzc{S}(x^p) \coloneqq  \diag(\begin{bmatrix} \sqrt{{x^p_1}^2 + \wp^2}, \ldots, \sqrt{x^{p2}_{n^p} + \wp^2} \end{bmatrix}^\top)$, 
which translates into a version of \eqref{eqn:MD2_induced_primal_unconstrained} with a strategy-dependent coefficient.
\end{example}

\begin{example}(\textbf{Orthant})
	Let $\Omega^p = \mathbb{R}^{n_p}_{\geq 0}$ and consider $\vartheta^p(x^p) = \sum_{i = 1}^{n_p} x_i^p \log(x_i^p) - x_i^p$ with $0\log(0) = 0$. We can show $\nabla^2 \psi_\epsilon^p(x^p) = \epsilon [\diag({x^p})]^{-1}$ where $\psi_\epsilon^p = \epsilon \vartheta^p$. Using NG2, we immediately arrive at,	
	\begin{equation}
			\label{eqn:MD2_induced_nonneg_orthant}
			\dot \xi^p  =\beta(x^p - \xi^p) \quad 
			\dot x^p  =  \gamma\epsilon^{-1} \diag(x^p) (U^p(x) - \alpha (x^p - \xi^p)).
	\end{equation}  We note that the first-order variant $(\alpha = \beta = 0)$ of \eqref{eqn:MD2_induced_nonneg_orthant} was recently studied by \cite{Bervoets} as the  \textit{mean-dynamics} to a payoff-based learning dynamics in continuous games with non-negative orthant action sets. Suppose instead $\vartheta^p(x^p) = -\sum_{i = 1}^n \ln(x^p_i)$ (the log-barrier function), which leads to $\nabla^2 \psi_\epsilon^p(x^p) = \epsilon [\diag(x^{p2})]^{-1}$ and,
	\begin{equation}
		\label{eqn:log_barrier_dynamics}
		\!\dot \xi^p  =\beta(x^p\! -\! \xi^p) \quad 
		\!\dot x^p  =  \gamma\epsilon^{-1} \diag(x^{p2}) (U^p(x) \!-\! \alpha (x^p \!-\! \xi^p)).
	\end{equation} 
	We stress however that the log-barrier does not fall under our definition of a regularizer, as $\dom(\vartheta^p)$ is a proper subset of $\Omega^p$. \eqref{eqn:log_barrier_dynamics} with $\alpha = \beta = 0$ has been previously studied in the context of optimization, see \cite{Fiacco}.
\end{example}

\section{Discrete-Time Second-Order Dual Averaging with Noisy Observations}
\label{section:second_order_MD_noisy} 
So far we have considered a continuous-time setup whereby each player is able to acquire a partial-gradient $U^p$ at each time instance. In practical scenarios where the games are played in discrete-time, the acquired pseudo-gradient information could be corrupted due to a multitude of reasons, such as a noisy communication channel. This leads us to consider the so-called ``noise-corrupted" pseudo-gradient scenario (also known as ``semi-bandit" learning). In this setting, each player receives a realization of a so-called noise-corrupted version of the true pseudo-gradient, \[\widehat U^p_{k+1} \coloneqq U^p(x_k) + \zeta^p_{k+1},\] where $\zeta^p_k$ is some noise process. A special case of the semi-bandit scenario is when the payoff is obtained as the expectation of the true payoff, i.e.,  $\mathcal{U}^p(x_k^p; x^{-p}_k) = \mathbb{E}[\mathcal{U}^p(x^p_k; x^{-p}_k, v^p_k)]$, $v^p_k$ some random vector, where  $\mathbb{E}$ denotes the expectation operator. Here $\widehat U^p_k$ is an estimate of the expected partial-gradient $\nabla_{x^p} \mathbb{E}[\mathcal{U}^p(x^p_k; x^{-p}_k, v^p_k)]$. 

Let's consider the convergence of the discrete-time MD2 with noisy observations which we refer to as second-order \textit{dual averaging} or DA2, 
\begin{equation}
	\begin{cases}
		X^p_k &=  C^p_\epsilon(Z^p_k)\\
		Z^p_{k+1} & = Z^p_k + \gamma t_{k+1} (\widehat U^p_{k+1} - \alpha (X^p_k - \Xi^p_k)) \\
		\Xi^p_{k+1} & = \Xi^p_k+ \tau_{k+1} \beta (X^p_k - \Xi^p_k),
	\end{cases}
	\tag*{(DA2)}
	\label{eqn:MD2_D}
\end{equation}
where $X^p_k, Z^p_k, \Xi^p_{k}$ are the stochastic counter-part of (resp.) $x^p, z^p, \xi^p$ at time  $k$. $\{t_k\}_{k\in \mathbb{N}}$, $\{\tau_k\}_{k\in \mathbb{N}}$ denote deterministic sequences of  non-increasing step-sizes, assumed to be common for all players. $\alpha, \beta, \gamma > 0 $ are the auxiliary parameters as from before. When $\alpha = \beta = 0$, we refer to the resulting expression as the dual averaging with noisy observations, 
\begin{equation}
	\begin{cases}
		X^p_k &=  C^p_\epsilon(Z^p_k)\\
		Z^p_{k+1} & = Z^p_k + \gamma t_{k+1} \widehat U^p_{k+1}, \\
	\end{cases}
	\tag*{(DA)}
	\label{eqn:noisy_md}
\end{equation}
which coincides with the dual averaging scheme studied in \cite{MZ2019} for $\gamma = 1$. 

To analyze the convergence behavior of DA2, we impose the following set of regularity assumptions \cite{Benaim_1999, Benaim_Hofbauer_Sorin05}. 
\begin{assumption}(\textbf{$\ell^2\backslash\ell^1$-Summability and Diminishing Step-sizes})
	\label{assump:MD2_l2_summability}
	\begin{equation} 
		\begin{split}
			&	\textstyle \sum_{k \in \mathbb{N}} t_{k} = \infty  \quad \textstyle \sum_{k \in \mathbb{N}} t^2_{k} < \infty \quad \textstyle \lim_{k \to \infty} t_{k} = 0, \\
			& 	\textstyle \sum_{k \in \mathbb{N}} \tau_{k} = \infty \quad 	\sum_{k \in \mathbb{N}} \tau^2_{k} < \infty  \quad \lim_{k \to \infty} \tau_{k} = 0.
		\end{split}  
	\end{equation} 
\end{assumption}
\begin{assumption}(\textbf{$L^2$-Bounded Martingale Difference Noise})
	\label{assump:MD2_noise_assump}
	Assume $\{\zeta_k\}_{k \in \mathbb{N}}, \zeta_k = (\zeta^p_k)_{p \in \mathcal{N}}$ is a $L^2$-bounded marginal difference process adapted to the filtration $\{\mathcal{F}_k\}_{k \in \mathbb{N}}$: each $\zeta_k$ is a random vector that is measurable with respect to $\mathcal{F}_k$ for each $k$, where each $\mathcal{F}_k$ is the $\sigma$-field, i.e., $\mathcal{F}_k = \sigma(\Xi_0, Z_0, \zeta_0, \ldots, \zeta_k)$ and $\mathcal{F}_k \subseteq \mathcal{F}_{k+1}$. In particular, $\{\zeta_k\}_{k \in \mathbb{N}}$ satisfies,
	\begin{equation} 
	\E[\zeta_{k+1} | \mathcal{F}_k] = 0, \forall k \in \mathbb{N} \quad a.s.,  
\end{equation} 
and for some $\sigma \geq 0$, 
\begin{equation} 
	\E[\|\zeta_{k+1}\|_\star^2 | \mathcal{F}_k] \leq \sigma^2, \forall k \in \mathbb{N} \quad a.s. 
\end{equation} 
\end{assumption}


\begin{assumption}(\textbf{Bounded Iterates})
	\label{assump:MD2_bounded_iterates}
	\begin{equation} \textstyle \sup_k \|Z^p_k\| < \infty \quad \sup_k \|\Xi^p_k\| < \infty, \forall p. \end{equation} 
\end{assumption}

Finally, we impose ``global integrability" on MD2, i.e., MD2 has a complete vector field. To do so, we need to convert MD2 into a first-order system by defining: $\omega^p \coloneqq ( \xi^p,  z^p )$, which generates the following stacked system on $\mathbb{R}^{2n} \cong \mathbb{R}^n \times \mathbb{R}^n$, 
\begin{equation} 
\begin{split} 
	\dot \omega^p 
	& = \begin{bmatrix} \dot \xi^p \\ \dot z^p \end{bmatrix} = \begin{bmatrix} \beta(C_\epsilon^p(z^p) - \xi^p) \\ \gamma(U^p(C_\epsilon(z)) - \alpha(C_\epsilon^p(z^p) - \xi^p)) \end{bmatrix}, \\
	& =  \begin{bmatrix} \beta C_\epsilon^p(z^p) \\ \gamma(U^p(C_\epsilon(z)) - \alpha C_\epsilon^p(z^p)) \end{bmatrix} + \begin{bmatrix} -\beta \xi^p \\ \gamma \alpha \xi^p \end{bmatrix},
\end{split} 	
\end{equation} 
or equivalently, $
	\dot \omega^p =  F^p(\omega) = g^p(z) + A^p(\xi^p), $
where, $ A^p(\xi^p)  = \begin{bmatrix} -\beta \xi^p \\ \gamma \alpha \xi^p \end{bmatrix}$ and $g^p(z) = \begin{bmatrix}  \beta C_\epsilon^p(z^p) \\  \gamma(U^p(C_\epsilon(z)) - \alpha C_\epsilon^p(z^p))  \end{bmatrix}.$

We then proceed to impose the following assumption on the overall system,
\begin{equation} \label{eqn:overall_system_MD2} \dot \omega  = F(\omega), \omega = (\omega^p)_{p \in \mathcal{N}},  F = (F^p)_{p \in \mathcal{N}}. \end{equation}  

\begin{assumption}(\textbf{Global Integrability})
	\label{assump:global_integrability_MD2} 
	The vector field $F: \mathbb{R}^{2n} \to \mathbb{R}^{2n}$ of \eqref{eqn:overall_system_MD2} is continuous globally integrable, that is, for every initial condition $(\xi(0), z(0)) \in \mathbb{R}^{2n}$, the unique solution of \eqref{eqn:overall_system_MD2} is defined for all $t \in \mathbb{R}$.
\end{assumption}
\begin{remark}
	We note that global integrability of MD2 is satisfied whenever the following holds, 
\begin{itemize}
	\item[(i)] 	Suppose $C^p_\epsilon$ is bounded on $\mathbb{R}^{n_p}$ and $U^p$ is bounded continuous locally Lipschitz on $\range(C^p_\epsilon)$ for all $p \in \mathcal{N}$, then $F$ \eqref{eqn:overall_system_MD2} is bounded locally Lipschitz and hence continuous globally integrable. This follows from \cite{Benaim_1999}.  
	\item[(ii)] 	Suppose $U^p$ is continuous locally Lipschitz and both $C^p_\epsilon$ and $U^p\circ C_\epsilon$ are \textit{sublinear} for all $p$, that is, 	\begin{equation} 
		\limsup_{\|z\|_2\to \infty}\dfrac{\|C^p_\epsilon(z)\|_2}{\|z\|_2} < \infty,\quad 
		\limsup_{\|z\|_2\to \infty}\dfrac{\|U^p \circ C_\epsilon(z)\|_2}{\|z\|_2} < \infty,
	\end{equation} then $F$ is sublinear and hence continuous and globally integrable. This follows from \cite{Benaim_Faure12}. 
\end{itemize}
\end{remark} 

Under \autoref{assump:MD2_l2_summability} -- 8, DA2 can be shown to track the continuous trajectories generated by MD2 via stochastic approximation arguments \cite{Benaim_1999,Benaim_Hofbauer_Sorin05}.

\begin{thm}
	
	\label{thm:MD2_discrete_convergence}
	Let $\mathcal{G}$ be a concave game and assume that every interior NE is globally merely VS. Suppose that all players choose strategies according to DA2 and $C_\epsilon = (C_\epsilon^p)_{p \in \mathcal{N}}$ is induced by $\psi^p_\epsilon \coloneqq \epsilon \vartheta^p$, where $\vartheta^p$ is $\mathcal{C}^2$ and satisfies \autoref{assump:regularizer}(i). Suppose that the following assumptions hold, \begin{itemize} 
		\item \autoref{assump:MD2_l2_summability} ($\ell^2\backslash\ell^1$ Summability and diminishing step-sizes) 
		\item  \autoref{assump:MD2_noise_assump} ($L^2$-bounded Martingale difference noise),
		\item  \autoref{assump:MD2_bounded_iterates} (Bounded iterates) 
		\item  \autoref{assump:global_integrability_MD2} (Global integrability)
	\end{itemize}  then $X_k$ converges to an interior mere VSS of $\mathcal{G}$ almost surely. The same conclusions hold whenever $\vartheta^p$ satisfies \autoref{assump:regularizer}(ii) and $\Omega^p$ is compact $\forall p$.
\end{thm}

\begin{remark}
	The closest result to ours is \cite[Theorem 4.7]{MZ2019}, but for DA (i.e., $\alpha,\beta = 0$). \autoref{assump:MD2_l2_summability}, \autoref{assump:MD2_noise_assump} are identical to theirs and  \autoref{assump:global_integrability_MD2} encapsulates the Lipschitz continuous requirement for $U$. The major departure is \autoref{assump:MD2_bounded_iterates}, which does not hold under certain circumstances, such as convergence towards (boundary) strict VSS in finite games. Hence our result only deals with interior mere VSS in general. To account for these boundary VSS possibly involves an extension of the inductive shadowing argument used by \cite{MZ2019} which we leave for future work.  
\end{remark}

\begin{remark}
	In addition to DA \cite{MZ2019}, there are several other algorithms that converge in similar settings. The most notable example is the \textit{mirror-prox/optimistic mirror descent} algorithm, which converges to mere VSS in perfect-gradient feedback setting \cite{Optimistic_MD}. However, the authors of \cite{Optimistic_MD} noted that convergence in null-coherent saddle point problems (a game with a mere VSS) fails in the presence of noise. This is a key advantage of DA2 over mirror-prox. Another algorithm is the \textit{stochastic iterative Tikhonov method} of \cite{Shanbhag13}, which converges in the semi-bandit setting. However, \cite{Shanbhag13} requires the game to be strictly monotone, whereas DA2 does not require monotonicity. 
\end{remark}
\section{Simulations}
\label{section:simulations} 

In this section, we consider three illustrative examples. The first example is the RPS game with a non-negative payoff for ties as in \autoref{example:rps_game}. We provide convergence behavior of MD and MD2 as well as MDA towards mere and weak VSS. We then provide the convergence behavior of DA2. The second example concerns a wireless power control game previously studied in \cite{Tatarenko19}. We show that DA2 converges in this game under different noise assumptions and study the effect of the number of players. Finally, we provide an example involving a generative adversarial network (GAN), which admits a locally mere VSS. We show that DA2 also converges in this game under different noise assumptions. In lieu of exact estimation of the basin of attraction for \textit{locally mere} VSS, which is difficult, we deal with all such cases through appropriate initialization. 


\begin{example}(\textbf{RPS with non-negative payoff for ties})
	\label{example:RPS_simulation} 
	Consider the RPS game as discussed in \autoref{example:rps_game}. We set the game parameters to be $\mathpzc{l} = 0, \mathpzc{w} = 2$ and vary the tie payoff parameter $\varsigma \in [1, \mathpzc{w}) = [1, 2)$ to induce different VS properties on the NE $x^\star = (x^p)_{p \in \mathcal{N}}, x^p =  (1/3, 1/3, 1/3)$. We simulate both MD and MD2 at initial conditions $z(0) = \begin{bmatrix} 3 &  2 &  1 \end{bmatrix}^\top, \xi(0) = \mathbf{0}$.  $\gamma, \beta, \alpha$ are kept as $1$.
	
	We begin by contrasting the continuous-time MD and MD2. For $\varsigma = 1$, this game is merely monotone, $x^\star$ is globally merely VS, MD2 converges by \autoref{thm:main} while MD diverges (\autoref{fig:rps_varsigma_1}). For $\varsigma = 1.1$, the game is $0.1$-weakly monotone, but since $x^\star$ is only $0.1$-weak VSS therefore it can be considered a nearly mere VSS (see \autoref{remark:weak_vss}). In this case, MD2 still converges while MD approaches a heteroclinic orbit (\autoref{fig:rps_varsigma_1_1}). 
	
	\begin{figure}[htp!]
		\centering
		\includegraphics[draft = false, scale = 0.5]{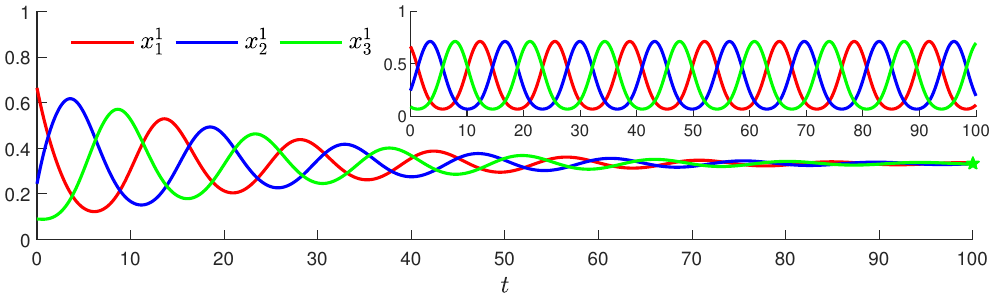}
		\caption{$\varsigma  = 1$, $\mathcal{G}$ merely monotone, $x^\star$ is merely VS. MD2 converges (Embedded: MD, cycling).}
				\label{fig:rps_varsigma_1}
	\end{figure}
	
	\begin{figure}[htp!]
		\centering
		\includegraphics[draft = false   , scale = 0.5]{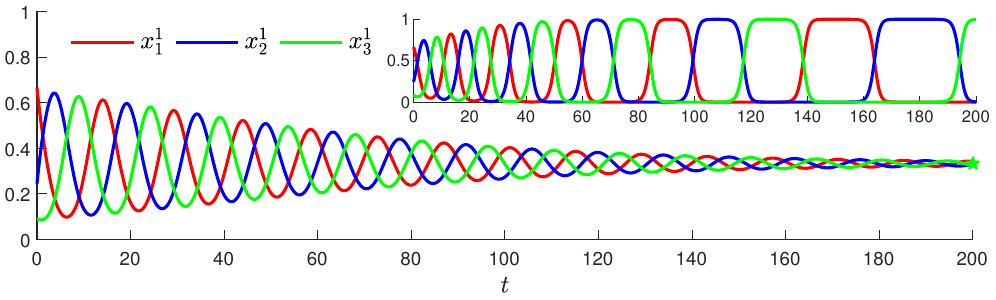}
		\caption{$\varsigma = 1.1$, $\mathcal{G}$ $0.1$-weakly monotone, $x^\star$ is $0.1$-weakly VS. MD2 converges (Embedded: MD, cycling).}
				\label{fig:rps_varsigma_1_1}
	\end{figure}

	It is useful to also examine the advantage of MD2 over time-averaged MD (MDA). While MDA does converge for $\varsigma = 1$ towards the mere VSS, when the equilibrium becomes just slightly weak ($\varsigma  = 1.1$), it no longer converges and instead approaches a Shapley triangle \cite{Unstable_Equilibria} (\autoref{fig:rps_varsigma_1_1_time_average}). This shows time-averaging is in general not a panacea to non-convergence.

	\begin{figure}[htp!]
		\centering
		\includegraphics[draft = false   , scale = 0.5]{	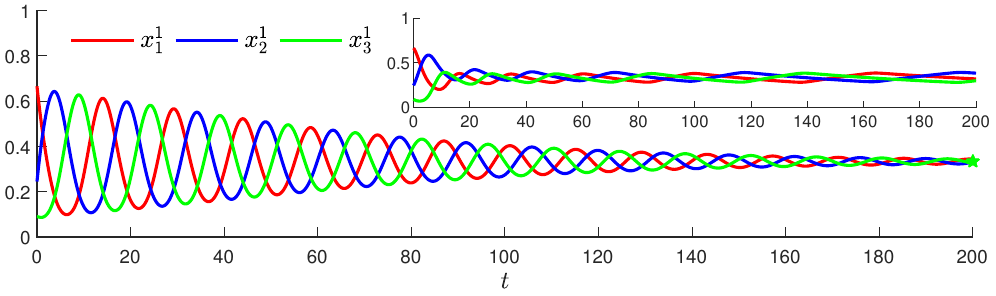}
		\caption{$\varsigma = 1.1$, $\mathcal{G}$ $0.1$-weakly monotone, $x^\star$ is $0.1$-weakly VS. MD2 converges (Embedded: MDA, cycling).}
		\label{fig:rps_varsigma_1_1_time_average}
	\end{figure}	
	
	Next, we perform a set of experiments for the globally merely VS case ($\varsigma = 1$) using DA2, where we assume the same initial condition as before. For each of the following simulations, we separately perturb the pseudo-gradient $U^p$ \eqref{example:rps_pseudo_grad_jacobian} with zero-mean Gaussian noise with the same variance $\sigma^2_\zeta > 0$ across all players. Such a perturbed gradient setting could describe a game involving multiple users interacting over a fully connected network with noisy communication channels. \autoref{fig:rps_variance_1} shows that DA2 easily converge in the low variance regimes $\sigma^2_\zeta = 1$. As we employ larger variance, e.g., $\sigma^2_\zeta = 10$, the standard step-sizes no longer leads to convergence: a larger $t_k$ will amplify the additive noise. Instead, we utilize step-size sequences of the form $\dfrac{c}{k^{\iota}}$ and search over both $\iota \in (0,1), c > 0$ for optimal sets of parameters. The simulation with tuned step-size is shown in \autoref{fig:rps_variance_10}. We see that despite the larger noise injected into $U^p$, the strategies $X = (X^p)_{p\in \mathcal{N}}$ still converge toward the vicinity of the mere VSS. 		
	\begin{figure}[htp!]
		\centering
		\includegraphics[draft = false, scale = 0.5]{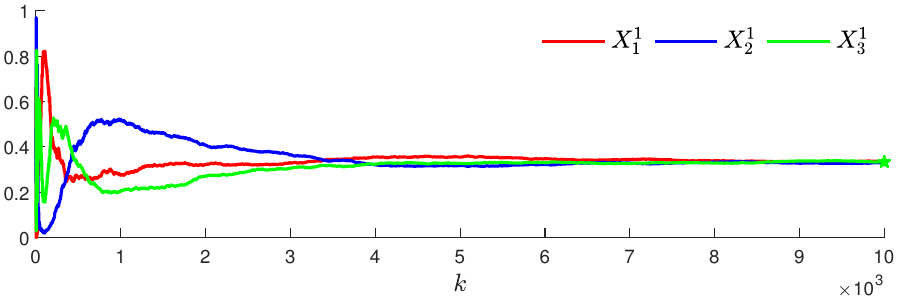}
		\caption{$\varsigma = 1, \sigma_\zeta^2 = 1$, $t_k = 4/k, \tau_k = 1/k$.} 
		\label{fig:rps_variance_1}
	\end{figure}

	\begin{figure}[htp!]
		\centering
		\includegraphics[draft = false, scale = 0.5]{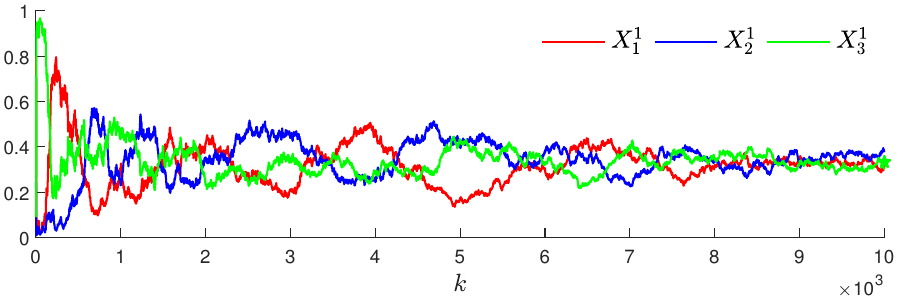}
		\caption{$\varsigma = 1, \sigma^2_\zeta = 10$, $t_k = 0.23/k^{0.48}, \tau_k = 0.34/k^{0.88}$.} 
		\label{fig:rps_variance_10}
	\end{figure}
\end{example}

\begin{example}(\textbf{Wireless power control game with a non-concave potential})
	\label{example:wireless} 
	Consider the wireless power control game in \cite{Tatarenko19}, where $\mathcal{N}= \{1, \ldots, N\}$ network users decide on \textit{intensities} $x^p \in \mathbb{R}$ of power flow to send over a wireless network. These intensities are converted to the transmitted power through an exponential function, i.e., $\exp(x^p) \in \mathbb{R}_{> 0}$. The payoff function for each user $p \in \mathcal{N}$ is modeled as,
	\begin{equation} 
		\textstyle \mathcal{U}^p(x^p; x^{-p}) = \log\left(1+ \dfrac{a^p \exp(x^p)}{1 + \sum_{p \neq r} a^r \exp(x^r)}\right) - \mathpzc{K}^p(x^p)
	\end{equation}
	where $a^p \in (0, 1], \forall p$ and $\mathpzc{K}^p(x^p) = b^p\log(1+\exp(x^p)) - c^p x^p$ is the cost of user $p$ for transmission, $b^p > 0, c^p \geq 0$. The partial-gradient $U^p$ is calculated to be,
	\begin{equation}
		\textstyle U^p(x) = \dfrac{a^p \exp(x^p)}{1 + \sum_{r \in \mathcal{N}} a^r \exp(x^r)} - \dfrac{b^p \exp(x^p)}{1+\exp(x^p)} + c^p.
	\end{equation}
	The second-order partial derivative of $\mathcal{U}^p$ can be calculated to be,
	\begin{equation}
	\dfrac{\partial^2 \mathcal{U}^p(x^p; x^{-p}) }{\partial x^q \partial x^p}  = \begin{cases}  \dfrac{-a^p a^q\exp(x^p + x^q)}{(1 + \sum_{r \in \mathcal{N}} a^r \exp(x^r))^2} & p \neq q \\[3ex] 		
		\begin{aligned} 
			& \dfrac{a^p\exp(x^p) (1 + \sum_{r \neq p} a^r \exp(x^r))}{(1 + \sum_{r \in \mathcal{N}} a^r \exp(x^r))^2}\\  
			& - \dfrac{b^p\exp(x^p)}{(1+\exp(x^p))^2} 
		\end{aligned}  & p = q
	\end{cases} 
	\end{equation}
	which generates a symmetric Jacobian $\Jacob_U(x), \forall x$,
	\begin{equation} {\mathbf{J}_U(x)}_{p, p} = 	\dfrac{a^p\exp(x^p)(1 + \sum_{r \neq p} a^r \exp(x^r))}{(1 + \sum_{r \in \mathcal{N}} a^r \exp(x^r))^2} - \dfrac{b^p \exp(x^p)}{(1+\exp(x^p))^2} \end{equation}
	and 
	\begin{equation}
		{\mathbf{J}_U(x)}_{p, q} = \dfrac{-a^p a^q \exp(x^p + x^q)}{(1+ \sum_{r \in \mathcal{N}} a^r \exp(x^r))^2}, p \neq q,
	\end{equation}
	hence the game is a potential game with a non-concave potential function $P$ such that $\nabla P = U$.  
	
    Consider an example with $N =2$ and problem parameters, $a = (a^1, a^2) = (1,1), b =( b^1, b^2) = (4, 4), c = (c^1, c^2) = (3, 3)$. For all players, $Z^p(0)$ is sampled uniformly from $[0, 10]$, $\Xi^p(0) = 0$. The NE of this game can be found at $x^\star = (1.8663, 1.8663)$. The Jacobian at the NE is $\Jacob(x^\star) = \diag(\begin{bmatrix} -4.308 & 0 \end{bmatrix}^\top)$, hence by \autoref{prop:second_order_char}, the NE is a locally mere VSS. By \autoref{thm:MD2_discrete_convergence}, DA2 converges towards this NE. For each of the following experiments, we simulate for $T = 10^4$ steps and report the strategies at termination. We restrict $\Omega^p$ to be a large but compact set $[-1000,1000]^2$ and set $C^p_\epsilon$ to be the projection operator, $C^p_\epsilon(Z^p) = \text{argmin}_{y^p \in \Omega^p} \|\epsilon^{-1} Z^p - y^p\|_2^2$  with $\epsilon = 1$. Unless specified otherwise, all other parameters $\gamma, \alpha, \beta$ are kept as $1$.  All additive noise are zero-mean Gaussian with the same variance $\sigma^2_\zeta$.
	
	We start our experiment with a small variance  $\sigma^2_\zeta  =  0.1$ and set the step-sizes to be $t_k = 0.39/k^{0.26}, \tau_k = 0.12/k^{0.64}$. \autoref{figure:MD2_Noisy_0_1} shows that DA2 converges to the NE. We increase the variance to $\sigma^2_\zeta = 10$ and \autoref{figure:MD2_Noisy_10} shows similar observation.
		\begin{figure}[htp!]
		\centering 
		\includegraphics[draft = false   , scale = 0.5]{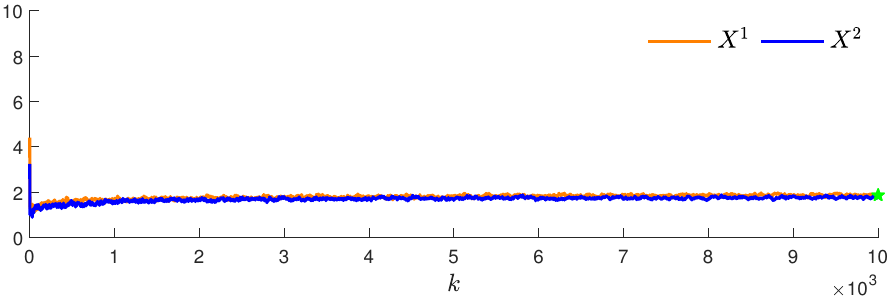}
		\caption{$\sigma^2_\zeta = 0.1$, $t_k = 0.39/k^{0.26}, \tau_k = 0.12/k^{0.64}$}
		\label{figure:MD2_Noisy_0_1}
	\end{figure}
	
	\begin{figure}[htp!]
		\centering 
		\includegraphics[draft = false, scale = 0.5]{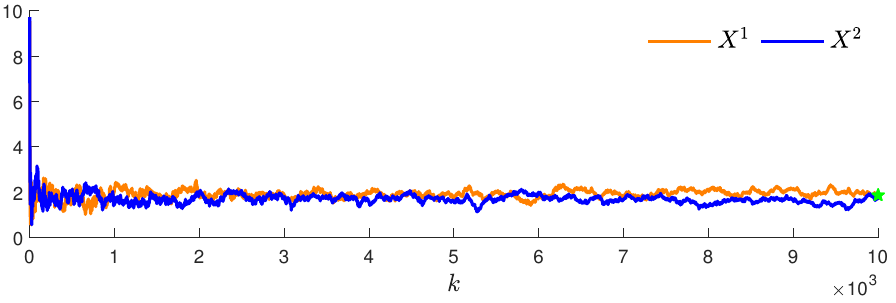}
		\caption{$\sigma^2_\zeta = 10$, $t_k = 0.38/k^{0.47}, \tau_k = 0.13/k^{0.71}$}
		\label{figure:MD2_Noisy_10}
	\end{figure}

	Next, we investigate some possible computational issues can arise when we increase the number of players in this game. Consider an example with $N = 10$ players. The game parameters are: \vspace{-0.1cm} 
	\begin{align*} & a \!= (12,\! 4, \!2, \!5, \!20, \!20, \!12, \!3, \!1, \!2),\\ & b \!= (15,\! 20,\! 20,\! 14,\! 21,\! 20,\! 16,\! 19,\! 17,\! 17),\\ & c \!= (13, \!12, \!14,\! 8, \!10,\! 19, \!10,\! 12, \!15, \!1). \end{align*} 
	A variance of $\sigma^2_\zeta = 1$ is applied for the following experiments. We employ step-size sequences  $t_k = \nicefrac{0.4}{k^{0.033}}, \tau_k = \nicefrac{0.78}{k^{0.001}}$ across all players. The NE of this game is located at $x^\star = (1.87, 0.41, 0.85, 0.28, -0.10, 16,0.51, 0.54, 2.01, -2.77)$ (rounded to two decimal places) and was confirmed to be a \textit{nearly} mere VSS (\autoref{remark:weak_vss}). We note that an exact mere VSS is difficult to produce for this game when there are a large number of players. 
	
	\begin{figure}[htp!]
		\centering 
		\includegraphics[draft = false   , scale = 0.5]{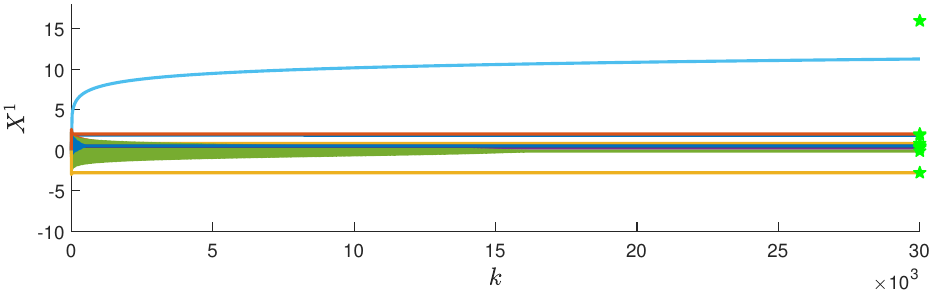}
		\caption{$\sigma^2_\zeta = 1$, $t_k = 0.40/k^{0.033}, \tau_k = 0.78/k^{0.001}$ ($X^2$ omitted).}
		\label{figure:Wireless_10p}
	\end{figure}
	\autoref{figure:Wireless_10p} shows that the strategy $X^1$ converges to $(1.87, 0.41, 0.85, 0.28, -0.10, 11.29, 0.51, 0.54, 2.01, -2.77)$ which is exactly the same as $x^\star$ in all entries except for the sixth entry ($16.51$ vs $11.29$), which is likely due to the small rate parameters employed in the step-sizes. This highlights the difficulty of employing the same step-size sequences across all players. We adjust all the parameters $\beta, \alpha$ and step-sizes $c k^{-\iota}, \iota \in (0,1), c > 0$ on a per-player basis, which resulted in a much closer convergence to the NE as opposed to applying them across all players uniformly  (\autoref{figure:Wireless_10p_better}).  
	
	\begin{figure}[htp!]
		\centering 
		\includegraphics[draft = false   , scale = 0.5]{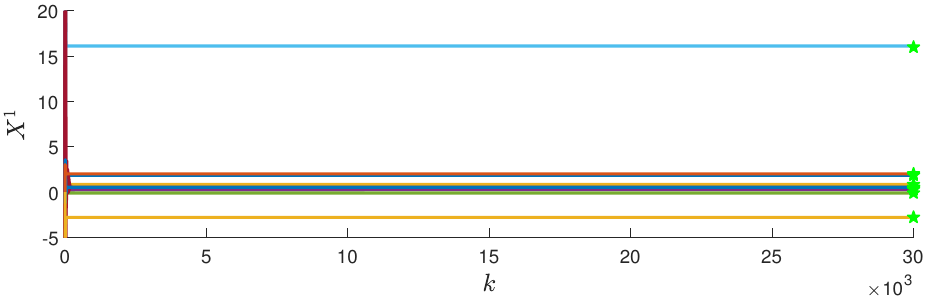}
		\caption{$\sigma^2_\zeta = 1$, per-player tuned parameters ($X^2$ omitted).}
		\label{figure:Wireless_10p_better}
	\end{figure}
\end{example}

\begin{example}(\textbf{Learning to Generate a Gaussian})
	
	Let $\mathpzc{Z} \sim \mathpzc{P}(\mathpzc{z})$ and $\mathpzc{X} \sim \mathpzc{Q}(\mathpzc{x})$ be two random variables. We wish to construct a model $\mathpzc{G}_\theta: \mathbb{R}^\mathpzc{r} \to \mathbb{R}^\mathpzc{m}$, $\mathpzc{r}, \mathpzc{m} \geq 1$, with an unknown, continuous parameter $\theta$ such that $\mathpzc{G}_\theta(\mathpzc{Z})$ recovers the statistics (mean, variance, etc.) of $\mathpzc{X}$. Following \cite{Daskalakis18}, $\mathpzc{G}_\theta$ can be constructed through solving, 
	\begin{equation}
		{\text{min}_{\theta}} \thinspace {\text{max}_{w}} \thinspace \mathbb{E}_{\mathpzc{X}}(\mathpzc{D}_w(\mathpzc{X})) - \mathbb{E}_{\mathpzc{Z}}(\mathpzc{D}_w(\mathpzc{G}_\theta(\mathpzc{Z}))), 
		\label{eqn:gan_objective}	
		\tag*{($\star$)}
	\end{equation} 
	where the model $\mathpzc{D}_w: \mathbb{R}^\mathpzc{m} \to \mathbb{R}$ is parametrized by an unknown continuous parameter $w$. This problem can be thought of as a game between the owners of the models $\mathpzc{G}_\theta$ and $\mathpzc{D}_w$ whereby $\theta, w$ are their respective strategies. The owner of $\mathpzc{G}_\theta$ (or the \textit{designer}) varies $\theta$ and uses $G_\theta(\mathpzc{Z})$ to estimate the statistics of $\mathpzc{X}$, whereas the owner of $\mathpzc{D}_w$ (or the \textit{tester}) varies $w$ in an attempt to incur the largest penalty possible for the discrepancy between $G_\theta(\mathpzc{Z})$ and $\mathpzc{X}$. 

	\begin{figure}[htp!]
		\centering 
		\includegraphics[scale = 0.7,  draft = false   ]{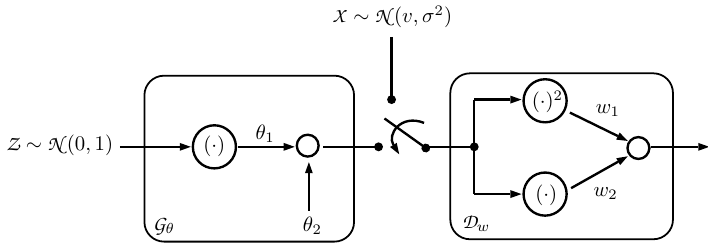}
		\caption{Diagram representation of $\mathpzc{G}_\theta$ and $\mathpzc{D}_w$ models}
	\end{figure}

	We construct a more complicated variant of the examples in \cite{Daskalakis18, Bo_Pavel_TechNote_2021}, whereby we assume that the $\mathpzc{G}_\theta$ owner wishes to learn the mean and variance of a one-dimensional Gaussian at the same time: let $\mathpzc{Z} \sim \mathpzc{N}(0,1)$, and $\mathpzc{X} \sim \mathpzc{N}(v, \sigma^2)$, where $v \in \mathbb{R}, \sigma \in \mathbb{R}_{> 0}$ are the mean and standard deviation of a Gaussian distribution. Let $\mathpzc{D}_w(\mathpzc{X}) = w_1 \mathpzc{X}^2 + w_2 \mathpzc{X}, w = (w_1, w_2) \in \mathbb{R} \times \mathbb{R}$ and $\mathpzc{G}_\theta(\mathpzc{Z}) = \theta_1 \mathpzc{Z}+\theta_2, \theta = (\theta_1, \theta_2) \in \mathbb{R} \times \mathbb{R}$, then a brief calculation shows that 
	\begin{equation*} (\star) = w_1((\sigma^2 + v^2) - (\theta^2_1 + \theta_2^2)) + w_2 (v - \theta_2). \end{equation*}
	Let $x^1 = \theta, x^2 = w$, we obtain a two-player ZS game, 
	\begin{equation}
		\textstyle\mathcal{U}^1(x^1; x^2) =  x^2_1(\sigma^2 + v^2 - \sum_{i = 1}^2 (x_i^1)^2) + x^2_2( v - x_2^1) 
	\end{equation}
	and $\mathcal{U}^2(x^2; x^1) =  -\mathcal{U}^1(x^1; x^2)$	
	where the action sets are, $\Omega^1 =  \mathbb{R} \times \mathbb{R} = \Omega^2$. The pseudo-gradient is, $
	U(x) = (-2{x^2_1} x^1_1, -2 {x^2_1} x^1_2 -  {x^2_2}, \sum_{i = 1}^2 (x_i^1)^2 - (\sigma^2 + v^2), x^1_2 - v),$
	which implies an interior NE at $x^\star = ({x^1}^\star, {x^2}^\star) = ((\sigma, v), (0, 0))$. The Jacobian of $U$ is shown to be, \begin{equation} \mathbf{J}_{U}(x) = 
	\begin{bmatrix} -2x_1^2 & 0 & -2x^1_1 & 0 \\ 
		0 & -2x^2_1 & -2x^1_2 & -1 \\ 
		\hphantom{-}2x^1_1 & \hphantom{-}2x^1_2 & 0 & 0 \\ 
		0 & 1 & 0 & 0 
	\end{bmatrix}. \end{equation} 
	Since  $y^\top \gJ_U(x) y= \frac{1}{2} y^\top(\mathbf{J}_{U}(x) + \mathbf{J}^\top_{U}(x))y = -2((y_1)^2 + (y_2)^2) x^2_1$ and $x^2_1$ is not restricted to be positive, therefore, this game is not monotone \cite[Prop. 2.3.2]{Facchinei_I}. However, at the interior NE, ${x^2}^\star = ({x_1^2}^\star, {x_2^2}^\star) = (0,0) \implies {x_1^2}^\star = 0$, and hence  $y^\top\gJ_U(x^\star)y = \frac{1}{2} y^\top(\mathbf{J}_{U}(x^\star) + \mathbf{J}^\top_{U}(x^\star))y = -2((y_1)^2 + (y_2)^2)(0)  = 0$. By \autoref{prop:second_order_char}(ii'), this implies that $x^\star$ is locally merely VS.  
	
	Consider an example with $v = 10$, $\sigma = 5$. We perform two simulations of the discrete-time MD2 with or without perturbation on the pseudo-gradient. We assume $Z(0)$ is uniformly sampled as follows: $Z^1(0) \sim \text{unif}([10, 50] \times [20, 40])$, $Z^2(0) \sim \text{unif}([0, 5] \times [0, 5])$, and $\Xi(0) = \mathbf{0}$. We restrict $\Omega^p$ to be a compact set $[-1000,1000]^2$ and set $C^p_\epsilon$ to be the projection operator as in \autoref{example:wireless}. We set $\lambda = \beta = \alpha = 1$. Each of the following simulations is ran for $T = 4 \times 10^4$ steps, with initial conditions randomly sampled as stated above. Note that our simulations has experimentally shown to accommodate larger sets of initial conditions than the ones provided above. 

First, we consider the noiseless case, i.e., $\sigma^2_\zeta = 0$. From \autoref{fig:GAN_Noiseless}, we see that $X^1$ converges to $X^{1\star}$ (highlighted using a green star). 	
\begin{figure}[htp!]
	\centering 
	\includegraphics[scale = 0.5, draft = false   ]{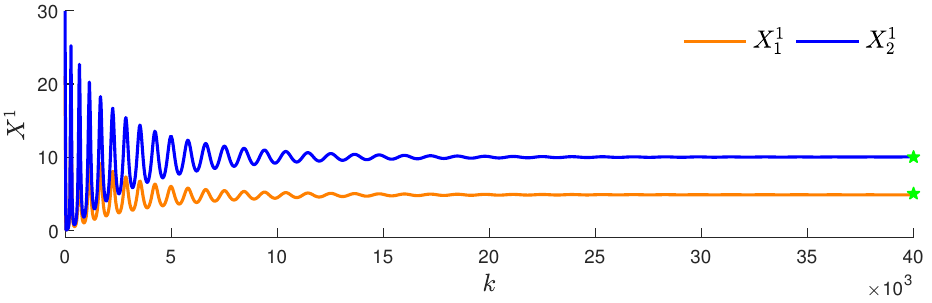}
	\caption{$\sigma^2_\zeta = 0$,  $t_k =\dfrac{0.0108}{k ^{0.3971}}$, $\tau_k = \dfrac{0.0485}{k^{0.7877}}$ ($X^2$ omitted).}
	\label{fig:GAN_Noiseless}
\end{figure}

Next, we consider the harder case where $U^p$ is subjected to zero-mean Gaussian noise with a very large variance $\sigma^2_\zeta = 30^2$ for each $p$. The trajectories of $X^1$ are shown in \autoref{fig:GAN_Noisy}. Our previous observation continues to hold. 
\begin{figure}[htp!]
	\centering 
	\includegraphics[scale = 0.5, draft = false   ]{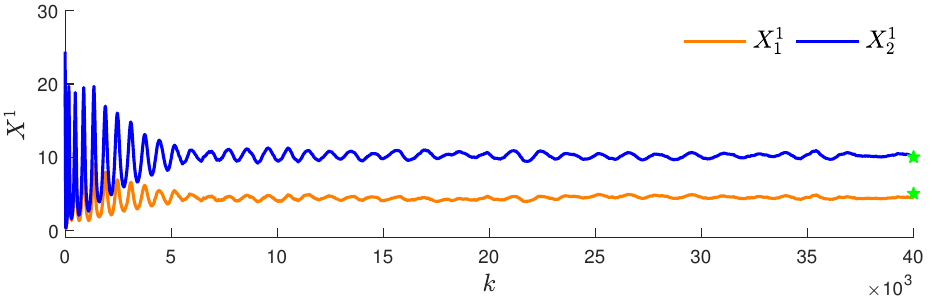}
	\caption{$\sigma^2_\zeta  = 30^2$,  $t_k =\dfrac{0.0108}{k ^{0.3971}}$, $\tau_k = \dfrac{0.0485}{k^{0.7877}}$ ($X^2$ omitted). }
	\label{fig:GAN_Noisy}
\end{figure}
\end{example}

\section{Conclusions}
\label{section:conclusion} 
In this paper, we have shown that MD2 overcomes many shortcomings of the first-order MD as well as other members of this class of dynamics, such as MD with time-averaging or discounting. To summarize our findings: first, we showed that the trajectories of MD2 converge exactly to an interior mere (and not necessarily strictly) VSS (\autoref{thm:main}). As such, it converges to interior mere NEs of monotone and pseudo-monotone games, and interior strict NE of quasi-monotone games (\autoref{cor:convergence_to_NE}). Through simulation, we also found that MD2 is robust around interior mere VSS (\autoref{remark:weak_vss}, \autoref{example:RPS_simulation}). Next, we provided a modification to MD2 based on vanishing perturbation which results in an exponential rate of convergence towards a strongly stable VSS (\autoref{thm:MD2_gamma}). We then showed MD2 can achieve no-regret while converging beyond a strict VSS, such as to an interior/fully-mixed NEs of zero-sum finite games (\autoref{thm:no_regret}). This result improves our current understandings  between no-regret learning and convergence \cite{Mertikopoulos18, Flokas}. We then provided several ways of deriving primal-space dynamics from MD2. Finally, we moved to the scenario where the game evolves in discrete-time while the pseudo-gradient is subjected to noise perturbation. Using stochastic approximation techniques, we were able to relate the limiting behavior of MD2 with that of its discrete-time, noisy counterpart, and showed that the discrete MD2 with noisy observations or second-order dual averaging (DA2) can converge to the interior mere VSS under standard regularity assumptions (\autoref{thm:MD2_discrete_convergence}). 

There are several outstanding issues that could be better understood. First, our work did not provide a thorough convergence proof for local VSS, which are dealt with appropriate initialization. Moreover, we did not address convergence towards boundary point for MD2, as our proofs relied on an \textit{interiority condition} associated with interior NE \eqref{eqn:MD2_rest_condition}. One possible approach for dealing with these boundary solutions is to adopt the approach of \cite{Mertikopoulos16}, by restricting ourselves to the context of finite games, where these boundary solutions are quite relevant. Another basic issue is that it is still not entirely clear to us at this time how MD2 relates to other existing continuous-time dynamics for optimization or games. Aside from the cases that we know of, e.g., \cite{Dian_LP_CDC2020}, it is possible that there are other existing algorithms that appear as specific instances of MD2. 

Our work opens up an extensive array of directions on the interface between dynamical systems and games. As a starting point, in the continuous-time, an interesting direction would be to explain the difference between MD and MD2 in zero-sum games through a volume compressibility perspective as in \cite{Giannou21, Coucheney}. We can also try to craft third or even higher-order versions of MD based on the technique in \cite{Bo_LP_TAC2020}. In the discrete-time case, we can further reduce the information, to that of the \textit{zero-order} information case (or full-bandit feedback). Moreover, we can tackle the case whereby the game itself has parameters that are time-varying. Finally, it is also worthwhile to examine MD2 in a discrete-time setting whereby the mirror map is replaced with a proximal operator.

\section{Appendix}
\renewcommand{\arraystretch}{1.2}
\begin{table}[h]
	\begin{tabular}{  |p{1.5cm}||p{6cm}|}
		\hline
		\textbf{Symbol} & \textbf{Definition}\\
		\hline
		\hline
		$x/x(t)/X_k$  & Strategy/at time $t$/at time $k$ \\
		\hline
		$z/z(t)/Z_k$ & Dual aggregate/at time $t$/at time $k$\\
		\hline
		$\xi/\xi(t)/\Xi_k$ & Primal aggregate/at time $t$/at time $k$ \\
		\hline
		$\mathcal{U}^p/U^p/U$ & Payoff/partial-gradient/pseudo-gradient\\
		\hline		
		$\mathbf{J}_U/\gJ_U$ & Jacobian of $U$/symmetric game Jacobian of $U$\\
		\hline
		$\vartheta^p/\psi^p_\epsilon/\psi^{p\star}_\epsilon$ & Regularizer/$\psi^p_\epsilon = \epsilon\vartheta^p$/convex conjugate of $\psi^p_\epsilon$\\
		\hline
		$C^p_\epsilon/C_\epsilon/\mathbf{J}_{C_\epsilon}$ & Player $p$'s/overall mirror map/Jacobian of $C_\epsilon$\\
		\hline
		$\gamma, \epsilon,\alpha,\beta$ & Parameters associated with MD/MD2/DA2\\
		\hline
		$\mathpzc{I}/\mathpzc{O}$ & Identity matrix/zero matrix \\
		\hline
	\end{tabular}
	\caption{A list of main notations used in this paper.}
	\label{table:1}  
\end{table}

The following two lemmas are standard in this area of literature, see \cite{Mertikopoulos16, Giannou21, Bo_Pavel_TechNote_2021}. In particular, \autoref{lem:mirror_map_properties_legendre} follows directly from the Legendre theorem \cite{Bauschke97}. 
\begin{lem}
	\label{lem:mirror_map_properties_legendre}
	Let $\psi^p_{\epsilon}\! \! \coloneqq \!\epsilon \vartheta^p\!$, $\epsilon \!>\!\! 0$, where  $\vartheta^p$ satisfies \autoref{assump:regularizer}(ii), and let ${\psi^p_{\epsilon}}^\star$ be the convex conjugate of $\psi^p_{\epsilon}$. 
	Then, 
	\begin{itemize}
		\item[(i)]  ${\psi^p_{\epsilon}}^\star\!:\! \mathbb{R}^{n_p} \!\to \!\mathbb{R}\!\cup\!\{\infty\}$ is closed, proper, strictly convex, steep, and finite-valued over $\mathbb{R}^{n_p}$.
		\item[(ii)]  	$\nabla \psi^p_{\epsilon}\! :\!  \interior(\dom(\psi^p_{\epsilon}))\!  \to\!  \interior(\dom({\psi^p_{\epsilon}}^\star))$ is a homeomorphism with inverse mapping $(\nabla \psi^p_{\epsilon})^{-1} \!=\! \nabla {\psi^p_{\epsilon}}^\star \!=\! C^p_{\epsilon}.$ 
		\item[(iii)]  $C^p_{\epsilon}$ is strictly monotone on $\interior(\dom({\psi^p_{\epsilon}}^\star))$. 
	\end{itemize}	
\end{lem}

\begin{lem}
	\label{lem:mirror_map_properties}
	Let $\psi^p_{\epsilon}\! \!\coloneqq \!\epsilon \vartheta^p\!$, $\epsilon \!>\!\! 0$, where  $\vartheta^p$ satisfies   \autoref{assump:regularizer}(i), and let ${\psi^p_{\epsilon}}^\star$ be the convex conjugate of $\psi^p_{\epsilon}$. 
	Then, 
	\begin{enumerate}
		\item[(i)] ${\psi^p_{\epsilon}}^\star\!:\!\mathbb{R}^{n_p} \!  \to \!\mathbb{R}\cup\{\infty\}$ is closed, proper, convex and finite-valued over $\mathbb{R}^{n_p}$, i.e., $\dom({\psi^p_{\epsilon}}^\star) = \mathbb{R}^{n_p}$.
		\item[(ii)]  ${\psi^p_{\epsilon}}^\star$ is continuously differentiable and $\nabla {\psi^p_{\epsilon}}^\star \!= \!C^p_{\epsilon}$. 
		\item[(iii)]  $C^p_{\epsilon}\!$ is surjective from $\mathbb{R}^{n_p}\!$ onto $\rinterior(\Omega^p)$ whenever $\psi^p_{\epsilon}\!$ is steep, and  onto $\Omega^p$ whenever $\psi^p_{\epsilon}\!$ is non-steep.  
	\end{enumerate}
\end{lem}

\begin{proof}(\textbf{Proof of \autoref{prop:second_order_char}})
		Let $ x, x^\prime \in \Omega$ be two arbitrarily chosen strategies. Consider $\overline x = \theta x + (1- \theta) x^\prime, \theta \in [0, 1]$. Then the proof for all these claims boil down to the equality condition found in  \cite[Prop. 1.44]{Ozdaglar}, in which it be can shown, 
		\begin{equation}
		\label{eqn:prop_second_order_char_step}
		\textstyle (x- x^\prime)^\top (U(x) - U(x^\prime)) =  \textstyle \int_{0}^1 (x^\prime - x)^\top \gJ_U(\overline x) (x^\prime - x) \,\,\mathrm{d}\theta
		\end{equation}
		Rearranging \eqref{eqn:prop_second_order_char_step}, we have,
		\begin{equation}
			\begin{split} 
			\textstyle	(x- x^\prime)^\top U(x)  = & \textstyle \int_{0}^1 (x^\prime - x)^\top \gJ_U(\overline x) (x^\prime - x)  \,\,\mathrm{d}\theta\\ & + 	(x- x^\prime)^\top U(x^\prime).
			\end{split} 
		\end{equation}
		Let $x^\prime = x^\star$, then $(x- x^\star)^\top U(x^\star) \leq 0$ and we have, 
		\begin{equation}
			\label{eqn:definiteness_condition}
			\textstyle	(x- x^\star)^\top U(x) \leq \int_{0}^1 (x^\star - x)^\top  \gJ_U(\overline x )(x^\star - x) \,\,\mathrm{d}\theta.
		\end{equation}
		By the definiteness assumptions of $\gJ_U(x)$ on $T_{\Omega}(x)$ for all $x \in \Omega$ and $\overline x = \theta x + (1-\theta)x^\star$ is contained in $\Omega$ for any $\theta \in [0, 1]$, \eqref{eqn:definiteness_condition} implies statements $(i), (ii), (iii)$.
		Next, suppose the definiteness condition on $\gJ_U(x^\star)$ holds for all $y \in T_{\Omega}(x^\star)$. Since $U$ is assumed to be continuously differentiable, therefore $y^\top \gJ_U(x) y$ is continuous for all $x \in \Omega$. Using the standard $\epsilon-\delta$ definition of continuity, this means for every $\epsilon > 0$ there exists a $\delta(\epsilon) > 0$ such that for all $x \in \Omega$, $\|x - x^\star\| < \delta(\epsilon) \implies |y^\top \gJ_U(x) y - y^\top \gJ_U(x^\star) y| < \epsilon$. The latter statement, $|y^\top \gJ_U(x) y - y^\top \gJ_U(x^\star) y| < \epsilon$, implies $y^\top \gJ_U(x) y < \epsilon + y^\top \gJ_U(x^\star) y$. Since this holds for every $\epsilon > 0$, therefore this is equivalent to $y^\top \gJ_U(x) y \leq y^\top \gJ_U(x^\star) y$, and hence $\gJ_U(x)$ shares the same definiteness property as $\gJ_U(x^\star)$ for all $x$ in some ball $\mathcal{D} = \{x \in \Omega | \|x - x^\star\| < \delta(\epsilon)\}$. Let $\overline x =\theta x + (1- \theta) x^\star$, $\theta \in [0,1]$, $x \in \mathcal{D}$. This implies \eqref{eqn:definiteness_condition} holds locally on the set of all $\overline x$, and $(i'), (ii'), (iii')$ follows. The isolation condition of $(i), (i^\prime)$ are proven in \cite[Prop. 2.7]{MZ2019}.
\end{proof}

\begin{proof}(\textbf{Proof of \autoref{thm:main}})
	(i) Consider the Lyapunov function,
	\begin{equation} 
		\hspace{-0.5cm} 
		\label{eqn:thm_main_lyapunov}
		\begin{split}
			\textstyle V(\xi, z) = & \dfrac{\alpha}{2\beta}\|\xi(t) - x^\star\|_2^2 \\ & +  \textstyle \gamma^{-1}\sum_{p \in \mathcal{N}} \psi^p_\epsilon({x^p}^\star) - {z^p}^\top {x^p}^\star + \psi_\epsilon^{p\star}(z^p).
		\end{split} 
	\end{equation} 
	
	Here, $V$ is composed of a quadratic distance term, which measures the progress of $\xi(t)$ towards $\xi^\star = x^\star$, added onto a function which consists the collection of all the terms associated with the Fenchel-Young inequality\cite[p. 88]{Beck17}, also known as the Fenchel coupling in \cite{MZ2019}. Since $\vartheta^p$ satisfies \autoref{assump:regularizer}(i), $V(\xi, z)$ is continuous, positive definite (due to the Legendre theorem), $V(\xi, z) = 0$ iff $\xi = x^\star$ and $z = C^{-1}_\epsilon(x^\star)$.  
	Taking the time-derivative of $V(\xi, z)$ along the solutions of MD2, using $x = C_\epsilon(z)$, $\nabla {\psi^p_\epsilon}^\star =  C^p_\epsilon$ and $\dot z = \gamma(U(x) - \alpha\beta^{-1} \dot \xi)$,  
	\begingroup
	\allowdisplaybreaks
	\begin{align*}
		\dot V(\xi, z)
		& = \!\dfrac{\alpha}{\beta}  (\xi - x^\star)^\top \dot \xi + \gamma^{-1}(C_\epsilon(z) - x^\star)^\top \dot z\\
		& = \!\dfrac{\alpha}{\beta}  (\xi -x^\star)^\top \dot \xi + (C_\epsilon(z) - x^\star)^\top (U(x) - \dfrac{\alpha}{\beta}\dot \xi)\\
		& = \!\dfrac{\alpha}{\beta}  (\xi - x^\star - x + x^\star)^\top \dot \xi + (C_\epsilon(z) - x^\star)^\top U(x) \\
		& = \!\dfrac{\alpha}{\beta}  (\xi - x)^\top \dot \xi + (x - x^\star)^\top U(x).
	\end{align*}
	Substituting in  $\dot \xi = \beta(x - \xi )$, we have, 
	\begin{align*}  
		\dot V(\xi, z) & =  -\alpha \beta^{-2} \|\dot \xi\|^2_2 + (x - x^\star)^\top U(x) \numberthis \label{eqn:MD2_proof_step} \\
		& \leq   -\alpha \beta^{-2} \|\dot \xi\|^2_2 \leq 0,
	\end{align*} 
	\endgroup
	where we have used $x^\star$ is a mere VSS. 
	Observe that $\dot V(\xi, z) = 0$ only if $\dot \xi = \mathbf{0} \Longleftrightarrow x = \xi$. By the Legendre theorem \cite{Bauschke97},  ${\psi^p_\epsilon}^\star$ is coercive, which implies $V(\xi,z)$ is coercive (radially unbounded) on $\mathbb{R}^n \times \mathbb{R}^n$ and hence all of its sublevel sets are compact. Let $\mathcal{D}_c = \{(\xi, z) \in \mathbb{R}^n \times \mathbb{R}^n| V(\xi,z) \leq c\}$ be a compact sublevel set for some $c > 0$.  Consider $\mathcal{E} = \{(\xi, z) \in \mathcal{D}_c |\dot V(\xi, z) = 0\} = \{(\xi, z) \in \mathcal{D}_c|\xi  = x = C_\epsilon(z)\}$. Let $(\xi(t), z(t))$ be some solution starting in $\mathcal{E}$, then $\xi = x = C_\epsilon(z) \implies \dot \xi = \dot x = \J_{C_\epsilon}(z)\dot z \implies \mathbf{0} = \J_{C_\epsilon}(z)\dot z$. By Legendre theorem and the $\mathcal{C}^2$ assumption, $\J_{C_\epsilon}(z)$ exists and is invertible, therefore $\dot z = \J_{C_\epsilon}(z)^{-1} \mathbf{0} = \mathbf{0}$. This means the largest invariant set contained in $\mathcal{E}$ is $\mathcal{S} = \{(\xi, z) \in \mathcal{E} |\dot z = \mathbf{0}\}$. By LaSalle's invariance principle \cite[Theorem 3.3]{Haddad}, any solution starting from $\mathcal{D}_c$ converges to $\mathcal{S}$ as $t \to \infty$. On $\mathcal{S}$,  $\dot z = \mathbf{0} \implies U(x)  = U(C_\epsilon(z)) = \mathbf{0}$, which means $x = C_\epsilon(z)$ converges to some interior NE (which we may denote as $x^\star$, since $x^\star$ is arbitrary), which by assumption is a globally mere VSS. Global convergence follows from coercivity of $V(\xi, z)$.
	
	(ii)  Suppose instead $\Omega^p$ is compact for all players $p$. Let $x^\omega$ be an $\omega$-limit point of $x(t) = C_\epsilon(z(t))$ of MD2. Since $x(t)$ is bounded for all $t \geq 0$,  the existence of $x^\omega$ follows from Bolzano-Weierstrass theorem.  Suppose that $x^\omega$ is not a globally mere VSS $x^\star$. Take $\mathcal{O}_\omega$ be an open ball around $x^\omega$. By definition, there exists a sequence $\{x(t_k)\}_{k \in \mathbb{N}}$, $x(t_k) \in \mathcal{O}_\omega$, converging towards $x^\omega$,  where $\{t_k\}_{k \in \mathbb{N}}$ is an increasing sequence of times. Following the technique in \cite{Staudigl17}, we wish to build an auxiliary sequence in $\mathcal{O}_\omega$ and show that $V$ is unbounded below along this sequence, thereby obtaining a contradiction. 
	
	First, note that since $\vartheta^p$ is $\rho$-strongly convex, $\psi^p_\epsilon \coloneqq \epsilon\vartheta^p$ is $\epsilon \rho$-strong convex, and hence by the conjugate correspondence theorem \cite[Theorem 5.26]{Beck17}, ${\psi^p_\epsilon}^\star$ is $(\epsilon\rho)^{-1}$-smooth, i.e., $C^p_\epsilon = \nabla {\psi^p_\epsilon}^\star$ is $(\epsilon\rho)^{-1}$-Lipschitz. Let $\tau > 0$ be a time increment, then, by $z(t_k+ \tau) - z(t_k) = \int_{t_k}^{t_k + \tau} \dot z(s) \mathrm{d}s = \gamma \int_{t_k}^{t_k + \tau} U(x(s)) - \alpha(x(s) - \xi(s))  \mathrm{d}s$, we have,	
	\begin{align} 
		& \|x(t_k + \tau) - x(t_k)\|  = \|C_\epsilon(z(t_k + \tau) - C_\epsilon(z(t_k))\| \\
		& \leq (\epsilon\rho)^{-1} \|z(t_k + \tau) - z(t_k)\|_\star\\
		&\textstyle  \leq (\epsilon\rho)^{-1} \gamma \|\smallint_{t_k}^{t_k + \tau} U(x(s)) - \alpha(x(s) - \xi(s)) \mathrm{d}s\|_\star\\
		&\textstyle  \leq (\epsilon\rho)^{-1} \gamma \smallint_{t_k}^{t_k + \tau} \|U(x(s)\|_\star +  
		\alpha \|x(s) - \xi(s)\|_\star \mathrm{d} s\\
		& \leq  (\epsilon\rho)^{-1} \tau \gamma \max_{x \in \Omega} (\|U(x) \|_\star + \alpha \|x(s) - \xi(s)\|_\star)  \leq b(\tau),
	\end{align}  
	where $b(\tau)$ is a $\tau$-dependent upper-bound. Let $\mathcal{B}_\tau(x(t_k)) = \{x \in \Omega  | \|x - x(t_k)\| \leq b(\tau), x(t_k) \in \mathcal{O}_\omega\}$ be a ball around $x(t_k)$. Since $\mathcal{O}_\omega$ is open, therefore there exists some $\delta >0$, such that $\mathcal{B}_\tau(x(t_k)) \subseteq \mathcal{O}_\omega$ for all $\tau \in [0, \delta]$ and hence $x(t_k + \tau) \in \mathcal{O}_\omega$. 
	
	Integrating \eqref{eqn:MD2_proof_step}, we have,
	\begin{align}  
		\label{eqn:MD2_Lyapunov_function}
		V(\xi(t), z(t)) & \textstyle =  V_0 - \smallint_{0}^t \alpha\|x - \xi\|^2_2 + (x - x^\star)^\top U(x) \mathrm{d}\tau, 
	\end{align} 
	where $V_0 \coloneqq V(\xi(0), z(0))$ is some constant. Since $x^\star$ is assumed to be a globally mere VSS, therefore $(x -  x^\star)^\top U(x) \leq 0, \forall x$. Tossing out $(x -  x^\star)^\top U(x) \leq 0, \forall x$, \eqref{eqn:MD2_Lyapunov_function},  we obtain,
	\begin{align*}  
		V(\xi(t), z(t)) & \textstyle \leq  V_0 - \smallint_{0}^t \alpha  \| x(\tau) - \xi(\tau)\|^2_2 \mathrm{d}\tau.
	\end{align*}
	Expressing this inequality in terms of the sequence $\{x(t_k)\}_{k \in \mathbb{N}}$ and $\{\xi(t_k)\}_{k \in \mathbb{N}}$, $-\| x(\tau) - \xi(\tau)\|_2 \leq 0$, and using the $\delta$ which we have previous found, we have,
	\begin{align*}  
		V(\xi(t_k + \delta), z(t_k + \delta)) & \textstyle \leq  V_0 -  \sum_{\ell = 1}^k \smallint_{t_\ell}^{t_\ell + \delta} \alpha  \| x(\tau) - \xi(\tau)\|^2_2 \mathrm{d}\tau 
	\end{align*}
	for some $k$. Since $\Omega$ is compact, $x$ is continuous, $\xi$ is a continuous function of $x$ \eqref{eqn:individual_xi_p_expression}, and $\xi$ is not eventually identically equal to $x$ (otherwise, using the rest point condition for MD2 \eqref{eqn:MD2_rest_condition} with our assumption that all interior NE are globally mere, we arrive at $x^\omega = x^\star$ hence a contradiction), $\|x - \xi\|_2^2$ achieves a lower-bound on an interval $[t_\ell, t_\ell + \delta]$. Let $c = \max \{ l > 0 | \| x(\tau) - \xi(\tau)\|_2^2 > l, \tau \in [t_\ell, t_\ell + \delta], \ell = 1, \ldots, k\}$.  This means, 
	\begin{align*}  
		V(\xi(t_k + \delta), z(t_k + \delta)) \textstyle \leq  V_0 -  \sum_{\ell = 1}^k \smallint_{t_\ell}^{t_\ell + \delta} \alpha  c \mathrm{d}\tau
		\leq  V_0 -\alpha c \delta k,
	\end{align*}
	which shows that $V$ is unbounded below as $k \to \infty$, a contradiction. 
\end{proof}

\begin{lem}
	\label{lem:bregman_fenchel_equivalence}
	Suppose $\vartheta^p: \mathbb{R}^{n_p} \to \mathbb{R}\cup\{\infty\}$ satisfies \autoref{assump:regularizer}(i) or (ii). Let $h = \sum_{p \in \mathcal{N}} \vartheta^p$, ${x^p}^\star \in \dom(\vartheta^p), x^p \in \dom(\partial \vartheta^p), z^p \in \partial \psi_\epsilon^p(x^p), \psi_\epsilon^p = \epsilon \vartheta^p, \forall p$,  then, \begin{equation} \textstyle D_{h}(x^\star, x) = \epsilon^{-1}  \sum_{p \in \mathcal{N}} \psi_\epsilon^p({x^p}^\star)\! -\! {\psi_\epsilon^p}^\star(z^p) \!-\! {z^p}^\top {x^p}^\star, \end{equation} where $x^\star = ({x^p}^\star)_{p \in \mathcal{N}}, x = ({x^p})_{p \in \mathcal{N}}$. 
\end{lem}

\begin{proof} \textbf{(Proof of \autoref{lem:bregman_fenchel_equivalence})}
	By definition,
	\begin{align} \hspace{-0.26cm} D_h(x^\star, x) & \textstyle =  \sum_{p \in \mathcal{N}} \vartheta^p({x^p}^\star) \!- \vartheta^p({x^p}) \!- \epsilon^{-1} {z^p}^\top({x^p}^\star \!- x^p) \\
		& \textstyle =  \sum_{p \in \mathcal{N}} \vartheta^p({x^p}^\star) + {\vartheta^p}^\star(\epsilon^{-1} z^p) - \epsilon^{-1} {z^p}^\top {x^p}^\star\\
		& \textstyle = \epsilon^{-1}  \sum_{p \in \mathcal{N}} \epsilon \vartheta^p({x^p}^\star) + \epsilon {\vartheta^p}^\star(\epsilon^{-1} z^p) - {z^p}^\top {x^p}^\star \\
		& \textstyle = \epsilon^{-1}  \sum_{p \in \mathcal{N}} \psi_\epsilon^p({x^p}^\star) + {\psi_\epsilon^p}^\star(z^p) - {z^p}^\top {x^p}^\star,
	\end{align}
	where we used $\psi_\epsilon^p({x^p}^\star)  = \epsilon \vartheta^p({x^p}^\star), z^p \in \partial \psi_\epsilon^p(x^p) = \epsilon \partial \vartheta^p(x^p),  {\psi_\epsilon^p}^\star(z^p) = \epsilon {\vartheta^p}^\star(\epsilon^{-1} z^p), {x^p} \in \partial {\psi_\epsilon^p}^\star(z^p)$. We note that ${\psi_\epsilon^p}^\star(z^p)  = \epsilon {\vartheta^p}^\star(\epsilon^{-1} z^p)$ follows from the conjugate property $\epsilon \vartheta^p(x^p) \Longleftrightarrow \epsilon \vartheta^{p\star}(\epsilon^{-1} z^p)$ \cite[p. 93]{Beck17}. 
\end{proof}

\begin{proof}(\textbf{Proof of \autoref{thm:MD2_gamma}}) Consider the following Lyapunov function, 
	\begin{align*} 
		\textstyle	V(\xi, \!z, \!\gamma)  & \textstyle \!=\! \dfrac{\alpha\gamma}{2\beta}\|\xi - x^\star\|_2^2 \!+ \! \sum_{p \in \mathcal{N}} \psi^p_\epsilon({x^p}^\star) - {z^p}^\top {x^p}^\star\! +\! \psi_\epsilon^{p\star}(z^p),
	\end{align*} 
	Then taking the time-derivative of $V(\xi, z, \gamma)$ along MD2$\gamma$, and using $\nabla {\psi^p_\epsilon}^\star = {C^p_\epsilon}$, $x^p = {C^p_\epsilon}(z^p)$, $\dot z = U(x) - \gamma \alpha \beta^{-1}\dot \xi$,  we have, 
	\begingroup
	\allowdisplaybreaks
	\begin{align*}
		& \dot V(\xi, z, \gamma) =  \dfrac{\alpha\dot \gamma}{2\beta}\|\xi \!-\! x^\star\|_2^2 \!+\! \dfrac{\alpha \gamma }{\beta}  (\xi - x^\star)^\top \dot \xi \!+\! (C_\epsilon(z) \!-\! x^\star)^\top \dot z\\
		& = \!\dfrac{\alpha\dot \gamma}{2\beta}\|\xi \!-\! x^\star\|_2^2 \!+\! \dfrac{\alpha\gamma}{\beta}  (\xi \!-\! x^\star \!-\! x \!+\! x^\star)^\top \dot \xi \!+\! (x\! - \!x^\star)^\top U(x)\\
		& = \!\dfrac{-\alpha \eta\gamma}{2\beta \epsilon}\|\xi - x^\star\|_2^2 -\dfrac{\alpha\gamma}{\beta^2}\|\dot \xi\|_2^2 + (x - x^\star)^\top U(x)\\
		& \leq  \! \dfrac{-\alpha \eta\gamma}{2\beta  \epsilon}\|\xi - x^\star\|_2^2   +(x - x^\star)^\top U(x)\\
		&\textstyle \leq \! \dfrac{-\alpha \eta\gamma}{2\beta \epsilon}\|\xi - x^\star\|_2^2  -  \eta (D_{h}(x^\star, x)+D_{h}(x,x^\star))\\
		&  \textstyle \leq \! \dfrac{-\alpha \eta\gamma}{2\beta \epsilon}\|\xi - x^\star\|_2^2  - \eta D_{h}(x^\star, x)\\
		&\textstyle = \! -\eta \epsilon^{-1} (\dfrac{ \alpha \gamma }{2\beta}\|\xi - x^\star\|_2^2  + \sum_{p \in \mathcal{N}} \psi_\epsilon^p({x^p}^\star)\! -\! {\psi_\epsilon^p}^\star(z^p) \!-\! {z^p}^\top {x^p}^\star )\\
		&= \textstyle	 -\eta \epsilon^{-1}V(\xi, z, \gamma), \numberthis \label{eqn:time_varying_V_dot}
	\end{align*} 
	where we have used $D_{h}(x^\star, x) = \epsilon^{-1}  \sum_{p \in \mathcal{N}} \psi_\epsilon^p({x^p}^\star)\! -\! {\psi_\epsilon^p}^\star(z^p) \!-\! {z^p}^\top {x^p}^\star ), \forall x = C_\epsilon(z), z \in C_\epsilon^{-1}(x)$ (\autoref{lem:bregman_fenchel_equivalence}).  \eqref{eqn:time_varying_V_dot} implies $V(\xi, z, \gamma) \leq e^{-\eta \epsilon^{-1} t}V(\xi_0, z_0, \gamma_0)$. Substituting in the exact expression for $V(\xi, z, \gamma)$ and $V(\xi_0, z_0, \gamma_0)$, then rearranging the resulting expression, we obtain \eqref{eqn:MD2_rate}. \eqref{eqn:MD2_rate_sc} follows from $D_h(x^\star, x) \geq 2^{-1}\rho\|x^\star - x\|_2^2$ whenever $\vartheta^p$ is $\rho$-strongly convex.
	\endgroup
\end{proof}

\begin{proof}(\textbf{Proof of \autoref{thm:no_regret}})
	Without loss of generality, let $\gamma = 1$. Let $y^p \in \Omega^p$ be an arbitrary strategy. Since $\mathcal{U}^p(y^p; x^{-p}(t))$ is concave in $y^p$, $\mathcal{U}^p(y^p; x^{-p}(t)) \leq \mathcal{U}^p(x^p(t); x^{-p}(t)) + \nabla_{x^p}\mathcal{U}^p(x^p(t); x^{-p}(t))^\top(y^p - x^p(t))$ =  $\mathcal{U}^p(x^p(t); x^{-p}(t)) + U^p(x)^\top(y^p - x^p(t))$, $\forall t \geq 0$, therefore, the integral term of $\eqref{eqn:regret}$, can be written as (suppressing time-index for brevity), 
	\begin{align}
		& \textstyle \int_{0}^{t}   \mathcal{U}^p(y^p; x^{-p}) - \mathcal{U}^p(x)  \mathrm{d}\tau \leq \int_{0}^{t} ({y^p} - {x^p})^\top U^p(x) \mathrm{d}\tau,	\label{eqn:regret_proof_part_1} 
	\end{align} 
	Next, define, \begin{equation} \mathpzc{Q}^p(t) \coloneqq {\psi^p_\epsilon}^\star(z^p(t)) - {y^p}^\top z^p(t)  + \dfrac{\alpha}{\beta}(\dfrac{1}{2}\|\xi^p(t)\|_2^2 - { y^p}^\top {\xi^p(t)}). \numberthis \label{eqn:primal-dual-primal} \end{equation}  Taking the time-derivative of $\mathpzc{Q}^p$ along the solutions of MD2 and using $\nabla {\psi^p_\epsilon}^\star = {C^p_\epsilon}$, $x^p = {C^p_\epsilon}(z^p)$, $\dot z^p = U^p(x) - \beta^{-1}\alpha\dot \xi^p$, yields, 
	\begingroup
	\allowdisplaybreaks
	\begin{align} \dot{\mathpzc{Q}}^p(t) 
		&\!=\! \nabla {\psi^p_\epsilon}^\star(z^p)^\top \dot z^p  \!-\! {y^p}^\top {\dot z^p} + \dfrac{\alpha}{\beta}({\xi^p}^\top \dot \xi^p \!-\!{y^p}^\top \dot \xi^p)  \\ 
		& = {C^p_\epsilon}(z^p)^\top \dot z^p  - {y^p}^\top {\dot z^p} + \dfrac{\alpha}{\beta}({\xi^p}^\top \dot \xi^p  -{y^p}^\top \dot \xi^p)  \\ 
		& = ({x^p}- y^p)^\top \dot z^p + \dfrac{\alpha}{\beta}(\xi^p - y^p)^\top \dot \xi^p \\
		& = ({x^p} \!-\! y^p)^\top (U^p(x) \!-\! \dfrac{\alpha}{\beta}\dot \xi^p)+ \dfrac{\alpha}{\beta}(\xi^p \!-\! y^p)^\top \dot \xi^p \\
		& = ({x^p} - y^p)^\top U^p(x) + \dfrac{\alpha}{\beta}(-{x^p} + \xi^p)^\top \dot  \xi^p\\
		& =  ({x^p} - y^p)^\top U^p(x)  - \dfrac{\alpha}{\beta^2} \|\dot \xi^p\|_2^2.
		\label{eqn:regret_proof_part_2}
	\end{align} 
	\endgroup
	Carrying on from \eqref{eqn:regret_proof_part_1} and using \eqref{eqn:regret_proof_part_2}, we have,
	\begingroup
	\allowdisplaybreaks
	\begin{align} \hspace*{-0.1in}
		\textstyle \int_{0}^{t}  \mathcal{U}^p(y^p; x^{-p}) -\mathcal{U}^p(x)  \mathrm{d}\tau  
		&\textstyle  = \int_{0}^{t} - {\dot{\mathpzc{Q}}}^p(\tau) -  \dfrac{\alpha}{\beta^2} \|\dot \xi^p\|_2^2\mathrm{d}\tau\\
		& \leq  \mathpzc{Q}^p(0) - \mathpzc{Q}^p(t).
	\end{align}
	\endgroup
	Applying Fenchel's inequality \cite{Bauschke97} to the function $\mathpzc{Q}^p$, we obtain, $\mathpzc{Q}^p(t)  \geq -\psi^p_\epsilon(y^p)  - \dfrac{\alpha}{2\beta} \|y^p\|_2^2.$
	Hence $\mathpzc{Q}^p(0) - \mathpzc{Q}^p(t) \leq  \mathpzc{Q}^p(0) + \psi^p_\epsilon(y^p) +  \dfrac{\alpha}{2\beta} \|y^p\|_2^2 $, where $\mathpzc{Q}^p(0)$ is  some finite constant (due to \autoref{lem:mirror_map_properties}(i)).  Therefore we obtain the following inequality, 
	\begingroup
	\allowdisplaybreaks
	\begin{align}
		\hspace*{-0.3cm} \textstyle \int_{0}^{t}\!   \mathcal{U}^p(y^p; x^{-p})\! -\! \mathcal{U}^p(x)  \mathrm{d}\tau\
		&\! \leq \!  \mathpzc{Q}^p(0)\!  +\! \psi^p_\epsilon(y^p)\! +\!  \dfrac{\alpha}{2\beta} \|y^p\|_2^2.
		\label{eqn:regret_proof_part_3}
	\end{align} 
	\endgroup 
	Plugging into \eqref{eqn:regret_proof_part_3} the definition of regret, we obtain, 
	\begingroup
	\allowdisplaybreaks
	\begin{align*} 
		\mathpzc{R}^p(t)  & \textstyle = \max\limits_{y^p \in \Omega^p} t^{-1} \int_{0}^{t} \mathcal{U}^p(y^p; x^{-p}(\tau)) - \mathcal{U}^p(x(\tau))  \mathrm{d}\tau \\
		& \leq \max\limits_{y^p \in \Omega^p} t^{-1} (\mathpzc{Q}^p(0)  + \psi^p_\epsilon(y^p) +  \dfrac{\alpha}{2\beta} \|y^p\|_2^2). \numberthis 		\label{eqn:regret_proof_final_step}
	\end{align*} 
	\endgroup 
	Since $\Omega^p$ is assumed to be compact, and the numerator is continuous on all of $\Omega^p$, therefore by Weierstrass theorem \cite[Theorem 2.12]{Beck17} a maximum is achieved. Taking the limsup yields the desired result. When $\alpha = 0$, this proof recovers no-regret bound of MD \cite{Mertikopoulos18}. 
\end{proof}

\begin{proof} (\textbf{Proof of \autoref{prop:primal_dynamics_general}})
	By \autoref{lem:mirror_map_properties}(ii), $x^p = C_\epsilon^p(z^p) = \nabla {\psi_\epsilon^p}^\star(z^p)$, therefore $x^p \in \partial {\psi^p_\epsilon}^\star(z^p)$. Using the fact that $\partial {\psi_\epsilon^p}$ is the inverse map of $\partial {\psi^p_\epsilon}^\star$, for all $x^p \in \dom(\partial \psi^{p\star}_\epsilon) = \rinterior(\Omega^p)$, \begin{equation}z^p \in \partial \psi^p_\epsilon(x^p)\! =\! \{\nabla \psi^p_\epsilon(x^p)\} + N_{\Omega^p}(x^p) \label{eqn:sub_diff_psi_p}, \end{equation} 
	From \eqref{eqn:sub_diff_psi_p}, we can write $z^p = \nabla \psi^p_\epsilon(x^p) + n^p(x^p)$ where $n^p(x^p) \in N_{\Omega^p}(x^p)$. Plugging this expression into \eqref{eqn:primal_first_order_MD2_p}, we have, 
	\begin{equation}
		\hspace*{-0.25cm}
		\begin{cases}
			\dot \xi^p &\!  \in \beta(x^p - \xi^p)\\
			\dot x^p  &\! \in \gamma \nabla^2 {\psi^p_\epsilon}^\star(\nabla \psi^p_\epsilon(x^p)\!+\!n^p(x^p)) (U^p(x)\! -\! \alpha (x^p \!-\! \xi^p)).  
		\end{cases} 
		\label{eqn:primal_first_order_MD2_inclusion}
	\end{equation}
	
	To reduce \eqref{eqn:primal_first_order_MD2_inclusion} from a set of differential inclusions to a set of ODEs, we will show, for all $x^p \in \rinterior(\Omega^p)$,   \begin{equation} \nabla^2 {\psi^p_\epsilon}^\star(\nabla \psi^p_\epsilon(x^p)+n^p(x^p)) =  \nabla^2 {\psi^p_\epsilon}^\star(\nabla \psi^p_\epsilon(x^p)). \end{equation} Indeed,  $x^p \in \partial \psi^{p\star}_\epsilon(z^p) \Longleftrightarrow z^p \in \partial \psi^{p}_\epsilon(x^p) \overset{\eqref{eqn:sub_diff_psi_p}}{\Longleftrightarrow} \nabla \psi^p_\epsilon(x^p) + n^p(x^p) \in   \partial \psi^{p}_\epsilon(x^p)$. The last equation, $\nabla \psi^p_\epsilon(x^p) + n^p(x^p) \in   \partial \psi^{p}_\epsilon(x^p)$ means, $\forall y^p \in \Omega^p$
	$\psi^{p}_\epsilon(y^p) \geq \psi^{p}_\epsilon(x^p) + (\nabla \psi^p_\epsilon(x^p) + n^p(x^p))^\top(y^p - x^p)$. Since $n^p(x^p) \in N_{\Omega^p}(x^p)$, therefore $n^p(x^p)^\top(y^p - x^p) \leq 0$, hence it is also true that $\psi^{p}_\epsilon(y^p) \geq \psi^{p}_\epsilon(x^p) + \nabla \psi^p_\epsilon(x^p)^\top (y^p - x^p)$ or equivalently, $\nabla \psi^p_\epsilon(x^p) \in \partial \psi^p_\epsilon(x^p)$, or $\partial \psi^{p\star}_\epsilon(\nabla \psi^p_\epsilon(x^p)) = \partial \psi^{p\star}_\epsilon(\partial \psi^p_\epsilon(x^p))$. Hence by re-substituting in \eqref{eqn:sub_diff_psi_p} and using the single-valuedness of $\partial \psi^{p\star}_\epsilon$, we obtain, $\nabla {\psi_\epsilon^p}^\star(\nabla {\psi^p_\epsilon}(x^p) + n^p(x^p)) =  \nabla {\psi_\epsilon^p}^\star(\nabla {\psi^p_\epsilon}(x^p))$. Taking the Jacobian of the preceding equality yields our desired result. 
\end{proof}

The proof of \autoref{thm:MD2_discrete_convergence} requires the following definitions from \cite{Benaim_1999, Benaim_Faure12, Benaim_Hofbauer_Sorin05}. Recall that a \textit{semiflow} $\Phi$ on a metric space $(\mathcal{M},d)$ is a continuous map $\Phi: \mathbb{T} \times \mathcal{M} \to \mathcal{M}, (t, x) \mapsto \Phi(t,x) = \Phi_t(x)$, where $\mathbb{T} =  \mathbb{R}_{\geq 0}$, such that, $\Phi_0 = Id$ and $\Phi_{t+s} = \Phi_t \circ \Phi_s$, for all $(t,s) \in \mathbb{R}_{\geq 0} \times \mathbb{R}_{\geq 0}$. If $\mathbb{T} = \mathbb{R}$, then $\Phi$ defines a flow. We say that a continuous function $f: \mathbb{R}_{>0} \to \mathcal{M}$ is an \textit{asymptotic pseudotrajectory} of $\Phi$ if $\lim\limits_{t \to \infty} \sup_{0 \leq h \leq T} d(f(t+h), \Phi_h(f(t))) = 0$ for any $T > 0$. A set $\mathcal{A} \subset \mathcal{M}$ is \textit{positively invariant} if $\Phi_t(\mathcal{A}) \subset \mathcal{A}$ for all $t \geq 0$ and \textit{invariant} if $\Phi_t(\mathcal{A}) = \mathcal{A}$ for all $t \in \mathbb{T}$.  

\begin{proof} (\textbf{Proof of \autoref{thm:MD2_discrete_convergence}}) For simplicity, we assume that all constants associated with MD2 are set to $1$. We break up our proof into the following steps:
	\begin{enumerate} 
		\item We first need to show that DA2 is the stochastic approximation of MD2.

		Rearranging the $Z^p_k$ update for DA2, $\dfrac{Z^p_{k+1}  -  Z^p_k}{ t_{k+1}} =   \widehat U^p_{k+1} = U^p(X_k) + \zeta^p_{k+1}$ and taking the expectation with respect to the filtration $\mathcal{F}_k = \sigma(\Xi_0, Z_0, \zeta_0, \ldots, \zeta_k)$ yields, 
		\begin{equation}
			\E(\dfrac{Z^p_{k+1}  -  Z^p_k}{ t_{k+1}}| \mathcal{F}_k) =  \E(U^p(X_k) + \zeta^p_{k+1} -  (X^p_k - \Xi^p_k)| \mathcal{F}_k),
		\end{equation}
		Since $\E(U^p(X_k)|\mathcal{F}_k) = U^p(X_k) - (X^p_k - \Xi^p_k)$ and $\E(\zeta^p_{k+1}|\mathcal{F}_k) = 0$ by \autoref{assump:MD2_noise_assump}, therefore, 
		\begin{equation}
			\E(\dfrac{Z^p_{k+1}  -  Z^p_k}{ t_{k+1}}| \mathcal{F}_k)  = U^p(X_k) - (X^p_k - \Xi^p_k),
		\end{equation}
		And similarly for $\Xi_{k+1}^p$, we have, 
		\begin{equation}
			\E(\dfrac{\Xi^p_{k+1}  -  \Xi^p_k}{ \tau_{k+1}}| \mathcal{F}_k)  = X^p_k - \Xi^p_k,
		\end{equation}
		which shows that MD2 is the mean ODE of DA2.
		
		\item Next, we need to build interpolated processes for $Z_k$ and $\Xi_k$ and show that the interpolated processes converge to a rest point of MD2. The full construction process is provided as follows: let  $\{\ell_k\}_{ k \in \mathbb{N}}$ be a sequence such that,
		\begin{equation}
			\textstyle \ell_0 = 0 \quad \ell_k = \sum_{i = 1}^k t_i, k \geq 1,
		\end{equation} and define the interpolated process, 
		\begin{equation}  \textstyle \underline Z^p(\ell_k + s) = Z^p_k + s\frac{Z^p_{k+1} - Z^p_k}{\ell_{k+1} - \ell_k},\end{equation} where $0 \leq s < t_{k+1}$. Similarly, let $\{\ell^\prime_k\}_{ k \in \mathbb{N}}$ be a sequence such that,
		\begin{equation}
			\textstyle \ell^\prime_0 = 0 \quad \ell^\prime_k = \sum_{i = 1}^k \tau_i, k \geq 1,
		\end{equation} and define the interpolated process associated with $\Xi^p$ as, \begin{equation} \textstyle \underline \Xi^p(\ell^\prime_k + s^\prime ) = \Xi^p_k + s^\prime\frac{\Xi^p_{k+1} - \Xi^p_k}{\ell^\prime_{k+1} - \ell^\prime_k},\end{equation} where $0 \leq s^\prime < \tau_{k+1}$. 
		\item Let  $\underline Z^p: \mathbb{R}_{\geq 0} \to \mathbb{R}^{n_p}$ and $\underline \Xi^p: \mathbb{R}_{\geq 0} \to \mathbb{R}^{n_p}$  denote the continuous functions (as functions of $s, s^\prime$ respectively) associated with the above processes and let $\underline Z = (\underline Z^p)_{p \in \mathcal{N}}, \underline \Xi = (\underline \Xi^p)_{p \in \mathcal{N}}$ be the stacked-vector of all the individual interpolated processes.
		\item  Next, define the overall process,
		\begin{equation}
			\underline W: \mathbb{R}_{\geq 0} \to \mathbb{R}^{2n} \quad \underline W = (\underline \Xi,  \underline Z), 
		\end{equation}  Define the interpolated process $\underline{W}$.  By Proposition 4.1 and Proposition 4.2 of \cite{Benaim_1999}, under \autoref{assump:MD2_l2_summability} - 8 $\underline W$ is an asymptotic pseudo-trajectory of the semiflow $\Phi: \mathbb{R}_{\geq 0} \times \mathbb{R}^{2n} \to \mathbb{R}^{2n}$ induced by $\dot \omega = F(\omega)$ \eqref{eqn:overall_system_MD2},  that is, 
		\begin{equation}
			\textstyle \lim_{t \to \infty} \sup_{0 \leq h \leq T} \|\underline W(t+h) -  \Phi(h, \underline W(t))\|_2 = 0,
		\end{equation}
		for any $T > 0$. 
		\item  By \autoref{assump:MD2_bounded_iterates}, $\underline W$ has compact closure, i.e., is \textit{pre-compact}.  Since we have shown that  $\underline W$ is a pre-compact APT of $\Phi$ induced by $\dot \omega$, by \cite[Theorem 5.7(i)]{Benaim_1999} the limit set \begin{equation} \textstyle L(\underline W) = \bigcap_{t \geq 0} \closure({ \underline{W}([t, \infty))}),\end{equation}  is internally chain transitive, which by \cite[Prop 5.3]{Benaim_1999}, $L(\underline W)$ is a compact invariant set.
		\item Recall that $V: \mathbb{R}^{2n} \to \mathbb{R}$ was used as the  Lyapunov function associated with  $\dot \omega = F(\omega)$ in \autoref{thm:main}. Consider the set of critical points of $V$,  $\mathcal{E} = \{(\xi, z) \in \mathbb{R}^{2n} | \dot V = 0\} = \{(\xi, z) \in \mathbb{R}^n \times \mathbb{R}^n | \xi = x = C(z)\}$. By the arguments in the proof of \autoref{thm:main}, since $V(\xi^\prime, z^\prime) < V(\xi, z)$ for all $(\xi, z) \in \mathbb{R}^{2n} \backslash \mathcal{E}$, $(\xi^\prime, z^\prime) = \Phi(t, (\xi, z))$ and $V(\xi^\prime, z^\prime) \leq V(\xi, z)$ for all $(\xi, z) \in \mathcal{E}, (\xi^\prime, z^\prime) = \Phi(t, (\xi, z))$, therefore $V$ is a Lyapunov function for $\mathcal{E}$ \cite{Benaim_Hofbauer_Sorin05}.  
		\item Since $V(\xi, z)$ is a Lyapunov function for $\mathcal{E} \subset \mathbb{R}^{2n}$, and $V(\mathcal{E})$ is a constant, which implies $\interior(V(\mathcal{E})) = \emptyset$, therefore by  \cite[Prop. 3.27]{Benaim_Hofbauer_Sorin05}, $L(\underline W)$ is contained in $\mathcal{E}$. By our assumption that every NE is a mere VSS, hence every point in $\mathcal{E}$ corresponds to an interior mere VSS and $L(\underline W) \subset \mathcal{E}$, therefore $L(\underline W)$  contains a compact subset of the rest points of $\dot \omega$.
		\item By the definition of a limit set, for any $\underline W(0)$, the interpolated process $\underline W(t)$ converges as $t \to \infty$. 
		\item From our construction of the interpolated process and the diminishing step-size assumption, i.e., $0 \leq s < \tau_{k+1}$, $0 \leq s^\prime < t_{k+1}$ and $\lim_{k\to \infty} t_k =  \lim_{k \to \infty} \tau_k = 0$, the convergence of the interpolated processes $\underline \Xi$ and $\underline Z$ implies the convergence of $\Xi_k$, $Z_k$ respectively. 
		\item By continuity of $C_\epsilon$, it follows that $X_k = C_\epsilon(Z_k)$ converges almost surely an interior mere VSS $x^\star$. 
	\end{enumerate}
\end{proof}

\nopagebreak
\begin{IEEEbiography}[{\includegraphics[width=0.98in]{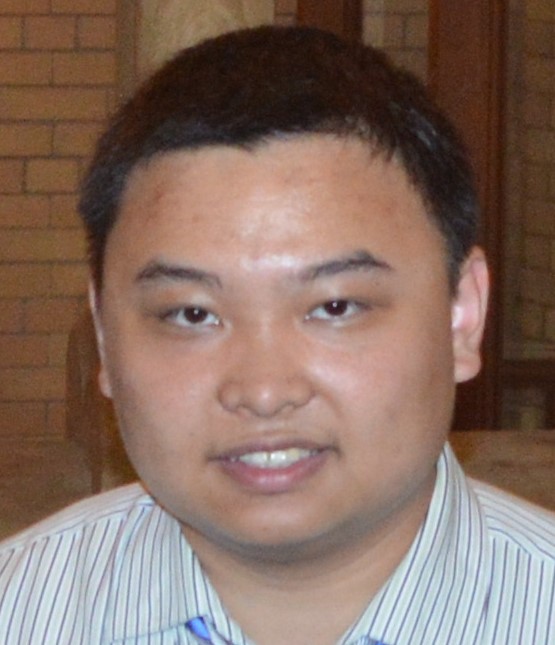}}]{Bolin Gao}
	received the B.A.Sc. (with Hons.) and M.A.Sc. degrees in electrical engineering in
	2015 and 2017, respectively, from the University of Toronto, Toronto, ON, Canada, where he is currently working toward the Ph.D. degree with
	the Systems Control Group. His research interests include design and control of game algorithms with applications to	reinforcement learning in multiagent systems.
\end{IEEEbiography}

\vspace{-15cm} 
\begin{IEEEbiography}[{\includegraphics[width=1in]{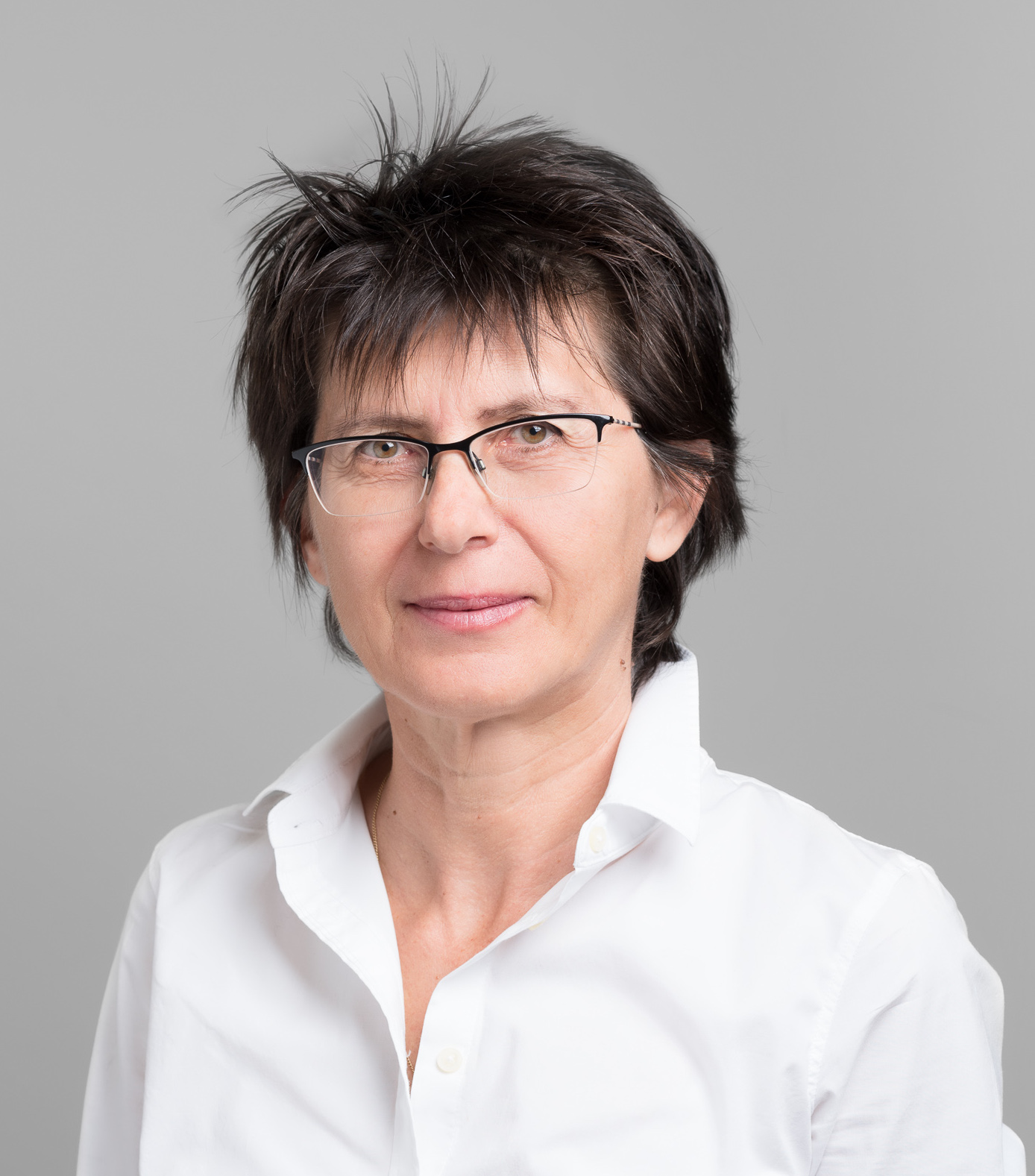}}]{Lacra Pavel}
	(M'92 - SM'04)  received the Dipl. Engineer from Technical University of Iasi, Romania and the Ph.D. degree in Electrical Engineering from Queen's University, Canada.  After a postdoctoral stage at the National Research Council and four years of industry experience, in 2002 she joined University of Toronto, where she is now a Professor in the Systems Control Group, Department of Electrical and Computer Engineering. Her research interests are in game theory and distributed optimization in networks, with emphasis on dynamics and control.  She is the author of  the book {\em Game Theory for Control of Optical Networks} (Birkh\"{a}user-Springer Science, ISBN 978-0-8176-8321-4, 2012). She is a Senior Editor of \emph{IEEE Transactions on Control of Network Systems}, a Senior Editor of the IEEE Open Journal of Control Systems and a Member of the Conference Editorial Board of the IEEE Control Systems Society; she acted as Publications Chair of the 45th IEEE Conference on Decision and Control.
\end{IEEEbiography}

\end{document}